\documentclass[a4paper]{article}

\usepackage[T1]{fontenc}
\usepackage{t1enc}

\usepackage{amsmath, amssymb, amsfonts,amsthm}
\usepackage{graphicx}
\usepackage{setspace}
\usepackage{commath}
\usepackage{pgf,tikz}
\usetikzlibrary{arrows}
\usepackage[hyphens]{url}
\usepackage{hyperref}
\usepackage{color}
\usepackage{setspace}
\usepackage{cite}
\usepackage{enumitem}
\usepackage{comment}
\usepackage{mathtools}
\usepackage{etoolbox}
\usepackage{quiver}
\usepackage{accents}
\usepackage[nottoc,numbib]{tocbibind}
\usepackage{tikz-cd}
\usepackage[normalem]{ulem}

\def\R{\mathbb{R}}

\def\N{\mathbb{N}}
\def\M{\mathbb{M}}

\def\A{\mathbb{A}}

\def\B{\mathbb{B}}

\DeclareMathOperator{\inter}{int}
\DeclareMathOperator{\stem}{stem}
\DeclareMathOperator{\dom}{dom}
\DeclareMathOperator{\add}{add}
\DeclareMathOperator{\cof}{cof}
\DeclareMathOperator{\cov}{cov}
\DeclareMathOperator{\non}{non}

\DeclareMathOperator{\rakovetkezo}{succ}

\newcommand{\bt}{\begin{theorem}}
\newcommand{\et}{\end{theorem}}
\newcommand{\bp}{\begin{proof}}
\newcommand{\ep}{\end{proof}}
\newcommand{\bd}{\begin{definition}}
\newcommand{\ed}{\end{definition}}
\newcommand{\bl}{\begin{lemma}}
\newcommand{\el}{\end{lemma}}
\newcommand{\bfa}{\begin{fact}}
\newcommand{\efa}{\end{fact}}

\newcommand{\s}{\subseteq}
\newcommand{\perf}{X_\textsl{perf}}
\newcommand{\ctbl}{X_\textsl{ctbl}}

\newcommand{\card}{\{\lambda^\omega \mid \lambda \textrm{ is a cardinal} \}\textrm{-universally Baire}}
\newcommand{\cardinf}{\{\lambda^\omega \mid \lambda \textrm{ is an infinite cardinal} \}\textrm{-universally Baire} }
\newcommand{\cech}{\textrm{\v{C}ech-complete-universally Baire}}
\newcommand{\topsp}{\textrm{topological-universally Baire}}
\newcommand{\n}{\{\mathbb{N}^\omega\}\textrm{-universally Baire}}
\newcommand{\p}{\textrm{Polish-universally Baire}}
\newcommand{\ch}{\textrm{compact Hausdorff-universally Baire}}
\newcommand{\compact}{\textrm{compact-universally Baire}}
\newcommand{\cantor}{\{2^\omega\} \textrm{-universally Baire}}
\newcommand{\cm}{\textrm{compact metrizable-universally Baire}}
\newcommand{\hausdorff}{\textrm{Hausdorff-universally Baire}}
\newcommand{\regular}{\{Y \mid Y \textrm{ has a regular open basis}\}\textrm{-universally Baire}}
\newcommand{\bs}{\mathbb{N}^\omega}

\newcommand{\br}{\begin{remark}}
\newcommand{\er}{\end{remark}}

\newcommand{\cat}{\mathbb{C} \mathbb{A} \mathbb{T}}

\newcommand{\customsection}[3]{%
  \refstepcounter{section}
  \sectionmark{#3}
  \addcontentsline{toc}{section}{\protect\numberline{\thesection}#2}
  \section*{\thesection\quad #1}
}

\definecolor{LinkSzin}{HTML}{000060}
\hypersetup{
  pdfborder = {5 3 1},
  colorlinks = true,
  linkcolor = LinkSzin,
  urlcolor = LinkSzin,
  citecolor = LinkSzin
}

\numberwithin{equation}{section}

\newtheorem{theorem}{Theorem}[section]
\newtheorem{lemma}[theorem]{Lemma}
\newtheorem{definition}[theorem]{Definition}

\newtheorem{ex}[theorem]{Example}
\newtheorem{claim}[theorem]{Claim}
\newtheorem{corollary}[theorem]{Corollary}
\newtheorem{question}[theorem]{Question}
\newtheorem{remark}[theorem]{Remark}
\newtheorem{fact}[theorem]{Fact}

\newtheorem*{diagram1}{Theorem~\ref{thm:diagram1}}

\newtheorem*{diagram2}{Theorem~\ref{thm:diagram2}}

\newtheorem*{compact=polish}{Theorem~\ref{t:equiv}}

\newtheorem*{nowherereflection}{Theorem~\ref{thm:non-Baire}}

\title{On various notions of universally Baire sets and reflection of non-Baire property in Cantor sets}

\author{M\'arton Elekes, M\'at\'e P\'alfy}

\begin{document}

\maketitle

\bgroup
\renewcommand{\thefootnote}{}%
\footnotetext{Both authors were supported by the National Research, Development and Innovation Office -- NKFIH, grants no.~146922,~124749,~129211,~152822. The second author was also supported by HAS Momentum Grant 2022-58.}
\egroup

\begin{abstract}
Universally Baire sets are a standard tool in modern set theory, where they are closely connected with forcing and absoluteness. In this paper we study universal Baireness instead as purely a regularity notion for subsets of Polish spaces. 
Our aim is not to develop the set-theoretic theory, but to place the notion on a systematic descriptive-set-theoretic footing and to clarify the several nonequivalent definitions that occur in the literature.

A subset of a Polish space is called \emph{universally Baire} if all its continuous preimages possess the property of Baire. However, it is important to specify what the domain of these continuous functions can be. In fact, a plethora of definitions are used by various authors in the literature, the class of possible domains is sometimes the class of Polish spaces, compact metrizable spaces, all topological spaces, compact topological spaces, etc. 

We clarify this situation by demonstrating that the above four options are pairwise nonequivalent under the Continuum Hypothesis, we determine all the implications between these four notions that are provable in $ZFC$, and also show that all the notions used in the literature are equivalent to one of these four.

Furthermore, we use Miller forcing to show that actually the cases of Polish and compact metrizable spaces are consistently equivalent. 

If the underlying Polish space (in which we consider whether a subset is universally Baire) is compact or even $\sigma$-compact, then instead of the above four essentially different notions there are only two, since it does not make any difference whether we consider Polish or compact metrizable spaces as domains, and similarly whether we consider arbitrary topological or compact topological spaces as domains.

We also classify all versions of the above definitions obtained by replacing the continuous preimages by Borel or Baire preimages.

We show that numerous definitions using Borel isomorphisms, changes of topologies and Boolean algebras are all equivalent to the case of Polish domains.

We present an argument that the notion using Polish domains is the natural dual to the notion of universally measurable sets.

Finally, as a byproduct, we prove that the following statement is independent of $ZFC$: if a subset of a Polish space has the property of Baire relative to every nowhere dense Cantor subspace, then the subset has the property of Baire in the whole space.
\end{abstract}

\begin{spacing}{1}
\setcounter{tocdepth}{2}
\tableofcontents
\end{spacing}

\section{Introduction}

\bigskip
Universally Baire sets have played a central role in modern set theory for several decades, particularly through their connections with forcing and generic absoluteness. At the same time, universal Baireness can be viewed purely as a regularity property of subsets of Polish spaces. The present paper adopts this latter, descriptive-set-theoretic point of view. 

Curiously, there is no consensus in the literature what the definition of a universally Baire set should be. In fact, a plethora of definitions are used by various authors, many of them (at least consistently) nonequivalent.


 The main aim of this paper is to clarify this situation. All the definitions that we have encountered in the literature have proven to be equivalent to one of the following form.

\begin{definition}
\label{def:univBaire}
Let $\mathcal{F}$ be a class of topological spaces, $X$ a Polish space and $A \subseteq X$. The set $A$ is called \emph{$\mathcal{F}$-universally Baire} if $f^{-1}(A)$ has the property of Baire in $Y$ for every $Y \in \mathcal{F}$  and for every continuous function $f: Y \longrightarrow X$.
\end{definition}

It is easy to see that $\mathcal{F}$-universally Baire subsets of $X$ are well-behaved for every $\mathcal{F}$, for example they form a $\sigma$-algebra that contains the Borel sets. Also, since the sets with the property of Baire are closed under the Suslin operation, $\mathcal{F}$-universally Baire sets are also closed under the Suslin operation, in particular $\mathcal{F}$-universally Baire sets contain the analytic sets. Another important property is that they are closed under continuous preimages between Polish spaces, that is, if $Y$ is Polish, $f:Y \longrightarrow X$ is a continuous function and $A\s X $ is $\mathcal{F}$-universally Baire, then $f^{-1}(A)$ is $\mathcal{F}$-universally Baire in $Y$. 

It should be mentioned that the very first definition of universal Baireness that we are aware of is due to J. P. R. Christensen \cite[pg. 24]{Chr} from 1974. His notion was formulated in slightly different terms, but it is in fact equivalent to $\mathcal{F}$-universally Baire with the choice when $\mathcal{F}$ is the class of all topological spaces (for this equivalence see Theorem \ref{Borel preimages theorem} below). In 1992 in their influential paper Feng-Magidor-Woodin \cite{FMW} found fundamental connections between this descriptive set theoretic notion and pure set theory. Namely, they have shown that certain absoluteness results (with respect to forcing) hold if and only if all members in certain projective pointclasses are $\mathcal{F}$-universally Baire with the choice when $\mathcal{F}$ is the class of all topological spaces again. For more on this connection see Section \ref{s:def} and \cite{FMW}.

Since 1974, numerous classes $\mathcal{F}$ have been used in the literature, ranging from compact metric spaces through \v{C}ech-complete spaces to spaces of the form $\lambda^\omega$, where $\lambda$ is an arbitrary infinite cardinal equipped with the discrete topology, just to list a few. It should be pointed out that the case of Polish domains will turn out to be equivalent to a notion often called $\omega$-universally Baire or $\aleph_0$-universally Baire in the set theory literature. For a more concise list of options see 
\cite{Chr, Fremlin, FMW, Woodin, Jab, Matheron, Re, Muller, Ikegami, Ikegami2, 
Steel, Neeman, Sar, 
Welch, FriSchi, LarShe, KetchLarZap, Khomsky, Wilson, Far, Tod, diPris, Lucke}.

Our first main result below will clarify this situation. We will show that there are four of these definitions such that all the known definitions are equivalent to one of these four, and under the Continuum Hypothesis (CH) these four definitions are pairwise nonequivalent.

It is important to point out that the next two main theorems (Theorems \ref{thm:diagram1} and \ref{thm:diagram2}) sum up numerous results, and many of these are not new. This part of the paper should be considered as a survey, invoking several results by Feng-Magidor-Woodin, Fremlin and Matheron-Solecki-Zelen\'y from \cite{Fremlin, FMW, Matheron}.

From now on, for specific cases of the class $\mathcal{F}$, e.g. for the classes of Hausdorff or metrizable spaces, we will freely use the terms \emph{Hausdorff-universally Baire} and \emph{metrizable-universally Baire}. 
But there will be four distinguished classes, which we highlight in the next definition.

\begin{definition}
A subset $A$ of a Polish space $X$ is called \emph{topological-universally Baire, Polish-universally Baire, compact-universally Baire, compact metrizable-universally Baire}, respectively, if it is $\mathcal{F}$-universally Baire with $\mathcal{F}$ equal to the class of all topological spaces, Polish spaces, compact topological spaces and compact metrizable spaces, respectively.
\end{definition}

\begin{diagram1}
For a subset of a Polish space, the following implications hold:

    \begin{tikzcd}[row sep = 1cm, column sep = -2 cm]
        & & \fbox{
\begin{tabular}{c} topological-universally Baire \\
$\Updownarrow$
\\
Hausdorff-universally Baire \\
$\Updownarrow$
\\
$\{Y \mid Y \text{ has a regular open basis} \}$-universally Baire\\
$\Updownarrow$
\\ \text{ \v{C}ech-complete-universally Baire} \\
$\Updownarrow$
\\
metrizable-universally Baire\\
$\Updownarrow$
\\
completely metrizable-universally Baire\\
$\Updownarrow$

\\ $\{ \lambda^\omega \mid \lambda \text{ is a cardinal} \}$-universally Baire \\
$\Updownarrow$
\\ $\{ \lambda^\omega \mid \lambda \text{ is an infinite cardinal} \}$-universally Baire \\
$\Updownarrow$
\\ $\{ Y \mid Y \text{ is an ordered space} \}$-universally Baire 
\end{tabular}} \arrow[Rightarrow, dr] \arrow[Rightarrow, dl] & \\
        & \fbox{ \begin{tabular}{c} Polish-universally Baire \\
$\Updownarrow$\\
$\{ \mathbb{N^\omega} \}$-universally Baire \end{tabular} }\arrow[Rightarrow, dr]
        & &  \fbox{\begin{tabular}{c}
        compact-universally Baire \\
        $\Updownarrow$
\\
        compact Hausdorff-universally Baire \end{tabular}} \arrow[Rightarrow, dl] \\
        & & \fbox{\begin{tabular}{c}
        compact metrizable-universally Baire \\
$\Updownarrow$
\\
        $\{ 2^\omega \}$-universally Baire
        \end{tabular}}  &
    \end{tikzcd} 

\smallskip
Moreover, under the Continuum Hypothesis the four notions in the above diagram are pairwise nonequivalent.
\end{diagram1} 

\newpage

Our second main result will show that the situation simplifies if the underlying Polish space is compact, or even $\sigma$-compact.

\begin{diagram2}
For a subset of a \emph{$\sigma$-compact} Polish space, the following implications hold:

\begin{center}
    \begin{tikzcd}[row sep = 0cm, column sep = 2cm]
     \fbox{
\begin{tabular}{c} topological-universally Baire \\
$\Updownarrow$
\\
Hausdorff-universally Baire \\
$\Updownarrow$
\\
$\{Y \mid Y \text{ has a regular open basis} \}$-universally Baire\\
$\Updownarrow$
\\ \text{ \v{C}ech-complete-universally Baire} \\
$\Updownarrow$
\\
metrizable-universally Baire\\
$\Updownarrow$
\\
completely metrizable-universally Baire\\
$\Updownarrow$

\\ $\{ \lambda^\omega \mid \lambda \text{ is a cardinal} \}$-universally Baire \\
$\Updownarrow$
\\ $\{ \lambda^\omega \mid \lambda \text{ is an infinite cardinal} \}$-universally Baire \\
$\Updownarrow$

\\ $\{ Y \mid Y \text{ is an ordered space} \}$-universally Baire \\
$\Updownarrow$
\\ compact-universally Baire\\
$\Updownarrow$
\\ compact Hausdorff-universally Baire \end{tabular}}
 \\
\Bigg\Downarrow
\\
\fbox{\begin{tabular}{c} Polish-universally Baire \\
$\Updownarrow$
\\ $\{ \mathbb{N^\omega} \}$-universally Baire \\
$\Updownarrow$
\\  compact metrizable-universally Baire \\
$\Updownarrow$
\\ $\{ 2^\omega \}$-universally Baire \end{tabular}}
    \end{tikzcd}
\end{center}

\smallskip
Moreover, under the Continuum Hypothesis the two notions in the above diagram are nonequivalent.
\end{diagram2}

It is a very natural question whether the consistent examples showing that the implications above cannot be reversed can be replaced by $ZFC$ examples. In most cases these are very intriguing open questions, but in one case we know that this is impossible.

\begin{compact=polish}
It is consistent with $ZFC$, since it holds in the Miller model, that a subset of a Polish space is Polish-universally Baire if and only if it is compact metrizable-universally Baire.
\end{compact=polish}

The key tool in the proof of the above theorem is the following new reflection principle that is interesting in its own right.

\begin{nowherereflection}

The following statement is consistent with $ZFC$, since it holds in the Miller model: If $X$ is a Polish space and a subset $A \subseteq X$ does not have the property of Baire, then there exists a nowhere dense Cantor set $C \subseteq X$ such that $A \cap C$ does not have the property of Baire in $C$.
\end{nowherereflection}

\begin{remark}
We remark that this reflection property is in fact independent of $ZFC$. Indeed, a so called Luzin set (see the next section for a definition) is a set without the property of Baire but it is countable in every nowhere dense subspace.
\end{remark}

After clarifying the relations between the versions of this definition, our next goal will be to identify the version that is the best dual notion to universal measurability. (Recall that a set in a Polish space is called \emph{universally measurable} if it is measurable with respect to the completion of every Borel probability measure.) These dualities have been prevalent in the field, starting from meager sets and nullsets, through sets with the property of Baire and measurable sets, to universally meager sets and universally null sets. In Section \ref{s:argument} we argue that the most natural choice for the definition of a universally Baire set should have been Polish-universally Baire. Most probably this is too late to be changed, since the definition has been coined for topological-universally Baire for decades in set theory texts, but still we would like to point out that Polish-universally Baire is the natural dual notion to universal measurability.

Finally, in Section \ref{s:open} we collect the open problems.

\section{Notations and preliminaries}

Most of the following definitions and facts can be found in \cite{Ke}.

A topological space $X$ is called a \emph{Polish space}, if it is separable and completely metrizable. When $\lambda$ is a cardinal, we also consider it as a topological space with the discrete topology and when we write $\lambda^\omega$, we consider $\lambda^\omega$ as a topological space equipped with the product topology. In the case $\lambda= \omega$, we will write $\bs$, which is called the \emph{Baire space} and the topological space $2^\omega$ is called the \emph{Cantor space}. (The reason why we use both $\N$ and $\omega$ is that we consider $\N$ as a discrete topological space and $\omega$ as a set without any further structure.) The Baire space and the Cantor space are Polish spaces.
A subspace of a Polish space is Polish if and only if it is a $G_\delta$ subspace.

We will denote the family of Borel sets in a topological space $(Y, \tau)$ by $\mathcal{B}(Y, \tau)$. When there is no danger of confusion, since there is only one topology involved, we will abbreviate $\mathcal{B}(Y, \tau)$ by $\mathcal{B}(Y)$. Similarly, if $A \subseteq Y$ is a subspace then we will sometimes also abbreviate $\mathcal{B}(A, \tau \vert_A)$ by $\mathcal{B}(A)$.

Let $Y$ be a topological space, then $A \s Y$ is called a \emph{nowhere dense set}, if for every nonempty open set $U \s Y$ there exists a nonempty open set $V \s U$, such that $V \cap A = \emptyset$. The set $A \s Y$ is called a \emph{meager set}, if it is the countable union of nowhere dense sets. The set $B \s Y$ is called \emph{comeager}, if $B^c$ is meager. For a topological space $(Y,\tau)$ we denote the meager subsets of $Y$ by $\mathcal{M}(Y, \tau)$. We will sometimes abbreviate $\mathcal{M}(Y,\tau)$ by $\mathcal{M}(Y)$.

For sets $A,B$ we denote by $A \triangle B$ the \emph{symmetric difference} of $A$ and $B$, that is $A \triangle B= (A \setminus B) \cup (B \setminus A)$. Let $Y$ be a topological space, a set $A$ has the \emph{property of Baire} if there exists an open set $U$ such that $U \triangle A$ is meager. The family of sets that have the property of Baire forms a $\sigma$-algebra, and it is the smallest $\sigma$-algebra containing all the open sets and all the meager sets. Moreover, every set with the property of Baire can be written as the union of a Borel set and a meager set. We will denote the family of sets with the property of Baire in a topological space $(Y, \tau)$ by $\mathcal{BP}(Y, \tau)$. We will sometimes abbreviate $\mathcal{BP}(Y, \tau)$ by $\mathcal{BP}(Y)$.

A function $f:Y \longrightarrow Z$ is \emph{Baire measurable} if whenever $U \s Z$ is open then $f^{-1}(U)$ has the property of Baire. We will also use the Baire category theorem, that is, if $Y$ is a complete metric space, then the countable intersection of dense open sets is dense.

A topological space $Y$ is said to be \emph{perfect}, if it has no isolated points. A subset of a topological space is called a \emph{perfect subset} (or simply perfect) if it is closed and perfect with respect to the subspace topology. If $X$ is a Polish space, then $X$ can be uniquely written as $X = X_\textsl{perf} \cup X_\textsl{ctbl}$, where $\perf$ and $\ctbl$ are disjoint and $\perf$ is a perfect subset of $X$ and $\ctbl$ is a countable open set. Later, we will use the following well-known lemma (see \cite[Proposition 2.1]{Miller}):

\begin{lemma}
\label{perfectcontbairelemma}
Every perfect Polish space contains a dense $G_\delta$ subspace homeomorphic to $\mathbb{N}^\omega$.
\end{lemma}

Let $X$ be a Polish space. The set $A \s X$ is called \emph{analytic} if there exists a Polish space $Y$ and a continuous $f:Y \longrightarrow X$ such that $f(Y)=A$. If $A \s X$ is an image of a Borel subset of a Polish space under a Borel map, then $A$ is analytic. The set $A \s X$ is called \emph{coanalytic} if $A^c$ is analytic.

Let $A$ be a nonempty set and $n \in \omega$.
The length of a finite sequence $s$ is denoted by $length(s)$.  If $s \in A^n$ and $m \leq n$, we denote by $s \vert_m$ the restriction of $s$ of length $m$, that is, if $s=(s_0,...,s_{n-1})$ then $s \vert_m=(s_0,...,s_{m-1})$. If $a \in A$ and $s \in A^n$, then $s^\smallfrown a$ denotes the sequence $(s_0,...,s_{n-1},a)$. We say that the finite sequence $t$ extends the finite sequence $s$ if there exists a $m \leq length(t)$ such that $t \vert_m =s$, in symbols $s \prec t$. We will denote by $A^{< \omega}$ the set of all finite sequences from $A$, that is, $A^{< \omega}= \bigcup\limits_{n \in \omega} A^n$. For $s,t \in A^{<\omega}$ we denote by $s \curlywedge t$ their longest common predecessor, that is $r=s \curlywedge t$ if and only if $r=s\vert_{\abs{r}}=t\vert_{\abs{r}}$ and $s(\abs{r}) \neq t(\abs{r})$ (or one of the two is undefined).

A \emph{tree} $T$ on a nonempty set $A$ is a subset $T \s A^{<\omega}$ such that $T$ is closed under initial segments, that is, if $t \in T$ and $s \prec t$, then $s \in T$. If $T$ is a tree on a set $A$ and $s \in A ^{<\omega}$, then $T[s]$ denotes the following tree:
$$
T[s]:= \{t \in T \vert t \prec s \textrm{ or } s \prec t\}.
$$

\underline{Convention}: Whenever we consider a topological space $Z$, we do \emph{not} assume that $Z$ is Hausdorff and we will carefully emphasize when the space $Z$ is assumed to be Hausdorff.

We will denote the \emph{interior, closure} and \emph{boundary} of a set $A$ by $\inter A$, $\overline{A}$ and $\partial A$, respectively.

Let $Z$ be a topological space and let $U \s Z$ be an open set, then $U$ is called a \emph{regular open set}, if $\inter \overline{U}=U$. We will say that $Z$ has a regular open basis, if it has a basis such that every element of the basis is a regular open set. If $A \s Z$ has the property of Baire, then there exists a regular open set $U$ and a meager set $M$ such that $A= U \triangle M$.

A topological space $Z$ is called \emph{\v{C}ech-complete}, if $Z$ is homeomorphic to a $G_\delta$ subset of a compact Hausdorff space.

A set $S \s [0,1]$ is called a \emph{Sierpi\'nski set} if it is uncountable, but has countable intersection with every Lebesgue null set, and $L \s \N^\omega$ is called a \emph{Luzin set} if it is uncountable, but has countable intersection with every meager set. Also note that both Sierpi\'nski and Luzin sets exist under the Continuum Hypothesis.


\customsection{A brief historical overview of the variants of the definition}{A brief historical overview of the variants of the definition}{A brief historical overview}\label{s:def}

Before turning to the list of published variants of the definition of universal Baireness, we need to warn the reader that the literature concerning this notion is rather chaotic.
Let us mention just a few peculiarities. First and foremost, there are roughly a dozen definitions, many of which are consistently nonequivalent. Moreover, as of the writing of this paper, Wikipedia (see \cite{Wik}) lists three definitions, claiming that they are equivalent, but in reality two of them are actually nonequivalent (at least under the Continuum Hypothesis).






Now, the goal of this section is to collect the definitions of universal Baireness that can be found in the literature.

As we have already mentioned above, for the very first such definition the credit should be given to Christensen.
Back in 1974 in \cite[pg. 24]{Chr} he defined a set $A \subseteq X$ to be universally Baire if $f^{-1}(A)$ has the property of Baire in $Y$ for every Hausdorff space $Y$  and for every Borel function $f: Y \longrightarrow X$.
As we will see in Theorem \ref{Borel preimages theorem} below, this is in fact equivalent to topological-universal Baireness. 

Then his definition was apparently forgotten for almost two decades. The next step was when Feng-Magidor-Woodin  reintroduced an equivalent notion in various different forms in 1992. Using our terminology, they showed that the following definitions are equivalent for a set $A \subseteq \mathbb{N}^\omega$ \cite[Cor. 2.1]{FMW}:

\begin{itemize}
\item 
$A$ is topological-universally Baire, 
\item 
$A$ is $\{ \lambda^\omega \mid \lambda \text{ is an infinite cardinal} \}$-universally Baire, 
\item
for every forcing notion $\mathbb{P}$ there are trees $T$ and $S$ such that $\mathrm{pr}(T) = A$, $\mathrm{pr}(S) = \mathbb{N}^\omega \setminus A$ and  $\mathbb{P} \Vdash \mathrm{pr}(T) \cup \mathrm{pr}(S) = \mathbb{N}^\omega$,
\end{itemize}
where $\lambda$ is considered as a discrete space and 
$\lambda^\omega$ is equipped with the product topology.



The special case of $\{ \mathbb{N}^\omega \}$-universally Baire sets was also introduced in \cite{FMW}, and it has generated substantial interest since then. The original terminology was $\omega$-universally Baire, or $\aleph_0$-universally Baire. It is included in our main theorem that this is equivalent to Polish-universal Baireness.

Rec\l aw in \cite{Re} introduced the following definition in 2000. A set $A \subseteq X$ is universally Baire if $f^{-1}(A)$ has the property of Baire in $\mathbb{R}$ for every Borel isomorphism $f: \mathbb{R} \longrightarrow X$. It will follow from Theorem \ref{thm: boole equivalence} that this is in fact also equivalent to Polish-universal Baireness.

Woodin in \cite{Woodin} uses compact Hausdorff-universal Baireness in 2002. While this notion turns out to be (consistently) nonequivalent to the Feng-Magidor-Woodin definition, in this paper only $\sigma$-compact underlying spaces are considered, where the above two notions coincide.

Matheron, Solecki and Zelen\'y in \cite{Matheron} introduced the class of Polish-universally Baire sets in 2006.

Fremlin in \cite{Fremlin} worked with the notion of 
\v{C}ech-complete-universal Baireness. This will also be included in our diagram, the main result will show that this is equivalent to topological-universal Baireness.


And finally, Jab{\l}o\'nska in \cite{Jab} applied compact metrizable-universally Baire sets in 2021.

\section{Universally Baire sets in arbitrary Polish spa\-ces}\label{univbairesection1}

The main result of the present section is the following theorem:

\bt
\label{thm:diagram1}
For a subset of a Polish space, the following implications hold:

    \begin{tikzcd}[row sep = 1cm, column sep = -2 cm]
        & & \fbox{
\begin{tabular}{c} topological-universally Baire \\
$\Updownarrow$
\\
Hausdorff-universally Baire \\
$\Updownarrow$
\\
$\{Y \mid Y \text{ has a regular open basis} \}$-universally Baire\\
$\Updownarrow$
\\ \text{ \v{C}ech-complete-universally Baire} \\
$\Updownarrow$
\\
metrizable-universally Baire\\
$\Updownarrow$
\\
completely metrizable-universally Baire\\
$\Updownarrow$

\\ $\{ \lambda^\omega \mid \lambda \text{ is a cardinal} \}$-universally Baire \\
$\Updownarrow$
\\ $\{ \lambda^\omega \mid \lambda \text{ is an infinite cardinal} \}$-universally Baire \\
$\Updownarrow$

\\ $\{ Y \mid Y \text{ is an ordered space} \}$-universally Baire
\end{tabular}} \arrow[Rightarrow, dr] \arrow[Rightarrow, dl] & \\
        & \fbox{ \begin{tabular}{c} Polish-universally Baire \\
$\Updownarrow$\\
$\{ \mathbb{N^\omega} \}$-universally Baire \end{tabular} }\arrow[Rightarrow, dr]
        & &  \fbox{\begin{tabular}{c}
        compact-universally Baire \\
        $\Updownarrow$
\\
        compact Hausdorff-universally Baire \end{tabular}} \arrow[Rightarrow, dl] \\
        & & \fbox{\begin{tabular}{c}
        compact metrizable-universally Baire \\
$\Updownarrow$
\\
        $\{ 2^\omega \}$-universally Baire
        \end{tabular}}  &
    \end{tikzcd} 

\smallskip
Moreover, under the Continuum Hypothesis the four notions in the above diagram are pairwise nonequivalent.
\et


\subsection{Auxiliary lemmas}

\begin{lemma}
\label{bővebbenuniveakkorkisebbenis}
Let $X$  and $\tilde{X}$ be arbitrary Polish spaces, $A \s \tilde{X} \s X$, and $\mathcal{F}$ be a class of topological spaces. If $A$ is $\mathcal{F}$-universally Baire in $X$, then $A$ is also $\mathcal{F}$-universally Baire in $\tilde{X}$.
\end{lemma}

\bp
We need to check that $f^{-1}(A)$ has the property of Baire in $Y$ for every $Y \in \mathcal{F}$ and continuous $f:Y \longrightarrow \tilde{X}$. But this is trivial, because every continuous $f:Y \longrightarrow \tilde{X}$ can be considered as a continuous $f:Y \longrightarrow X$.
\ep

\begin{lemma}
\label{bairetultranz}
Let $Z$ be a topological space and $A \s Y \s Z$. Suppose that $A$ has the property of Baire in $Y$ and $Y$ has the property of Baire in $Z$, then $A$ also has the property of Baire in $Z$.
\end{lemma}

\bp Since $A$ has the property of Baire in $Y$, we have that $A=B \cup M $, where $B$ is a Borel set in $Y$ and $M$ is a meager set in $Y$. Since $B$ is a Borel set in $Y$, there exists a Borel set $\tilde{B}$ in $Z$, such that $B=\tilde{B} \cap Y$. It is easy to see that if $N$ is a nowhere dense subset of $Y$, then it is a nowhere dense subset of $Z$, hence any meager set in $Y$ is also a meager set in $Z$. To finish the proof, we only need to note that $A = (\tilde{B} \cap Y) \cup M$, which clearly has the property of Baire in $Z$.
\ep

\subsection{Positive results}

In this subsection we will prove that all universally Baire notions that appear in the same frame in Theorem \ref{thm:diagram1} are indeed equivalent.

First we show that the following notions from the top frame are equivalent:

\begin{claim}\label{top box notions are the same}
For a subset $A \s X$, the following nine properties are equivalent:

\begin{enumerate}[label=(\arabic*)]
\item
\label{1}
$\topsp$,

\item\label{2}
\hausdorff,

\item\label{3}
$\regular$,

\item\label{4}
$\cech$,

\item\label{5}
metrizable-universally Baire,

\item\label{6}
completely metrizable-universally Baire,

\item\label{7}
$\card$,

\item\label{8}
$\cardinf$,

\item\label{9}
$\{ Y \mid Y \text{ is an ordered space} \}$-universally Baire.
\end{enumerate}

\end{claim}

\bp Recall that the following implications are all obvious because the corresponding $\mathcal{F}$-classes contain each other: $\ref{1} \Rightarrow \ref{2} \Rightarrow \ref{5} \Rightarrow \ref{6} \Rightarrow \ref{7} \Rightarrow \ref{8}$, $\ref{1} \Rightarrow \ref{3} \Rightarrow \ref{8}$, $\ref{1} \Rightarrow \ref{4} \Rightarrow \ref{8}$, $\ref{1} \Rightarrow \ref{9}$. So it is easy to see that all that remains to prove is \ref{9} $\Rightarrow$ \ref{8} $\Rightarrow$ \ref{1}.

In fact, we will prove \ref{9} $\Rightarrow$ \ref{8} $\Rightarrow$ \ref{7} $\Rightarrow$ \ref{4} $\Rightarrow$ \ref{1}. We start with the implication \ref{9} $\Rightarrow$ \ref{8}: Fix an infinite cardinal $\lambda$ and a continuous function $f:\lambda^\omega \longrightarrow X$ and our goal is to show that $f^{-1}(A)$ has the property of Baire. We define the following subset $R$ of $\lambda^\omega$ as
$$
R:=\{(\alpha_n)_{n\in \omega} \mid \exists^\infty n: \; \alpha_n \neq 0 \}.
$$
Since $R^c=\bigcup\limits_{k \in \omega} H_k$, where each $H_k=\{(\alpha_n)_{n\in \omega} \mid \forall n \geq k: \alpha_n =0 \}$ is closed and nowhere dense, the set $R$ is a comeager subset of $\lambda^\omega$.

We show that the order topology induced by the lexicographic order on $R$ and the subspace topology inherited from $\lambda^\omega$ coincide. Let us fix a point $p=(\alpha_n)_{n\in \omega} \in R$ and an open neighbourhood $V:=(\{\alpha_0\} \times \{\alpha_1\} \times...\times\{\alpha_l\}\times\lambda\times \lambda \times... )\cap R$ of $p$ in the subspace topology. Choose an $m > l$ such that $\alpha_m \neq 0$, then the interval $$U:= ((\alpha_0,...,\alpha_{m-1},0,\alpha_{m+1},\alpha_{m+2},...), (\alpha_0,...,\alpha_{m-1},\alpha_m + 1,\alpha_{m+1},\alpha_{m+2},...))$$ is contained in $V$ and contains $p$, showing that the topology induced by the lexicographic order on $R$ is finer than the subspace topology on $R$.

Conversely, suppose that we have a point $p=(\alpha_n)_{n\in \omega} \in R$ and an interval around it, that is, $(\alpha_n)_{n\in \omega} \in ((\beta_n)_{n\in \omega}, (\gamma_n)_{n\in \omega})$. Choose a sufficiently large $l$ such that $(\alpha_n)_{n\in l}, (\beta_n)_{n\in l}, (\gamma_n)_{n\in l}$ are pairwise distinct sequences of length $l$. Then it is easy to see that $p \in (\{\alpha_0\} \times \{\alpha_1\} \times...\times\{\alpha_{l-1}\}\times\lambda\times \lambda \times... )\cap R \s ((\beta_n)_{n\in \omega}, (\gamma_n)_{n\in \omega})$, which shows that the subspace topology on $R$ is finer than the topology induced by the lexicographic order.

Hence, if $g:=f \vert_R$ then the set $g^{-1}(A)=f^{-1}(A) \cap R$ has the property of Baire in $R$. By Lemma \ref{bairetultranz} $ f^{-1}(A) \cap R$ has the property of Baire in $\lambda^\omega$. Since $f^{-1}(A) \setminus R$ is meager in $\lambda^\omega$, we conclude that $f^{-1}(A)$ has the property of Baire in $\lambda^\omega$.

$\ref{8} \Rightarrow \ref{7}:$ Note that $\emptyset^\omega= \emptyset$, $|1^\omega|=1$, hence every set is $\{\emptyset^\omega\}$-universally Baire and $\{1^\omega\}$-universally Baire. Also note that instead of $\omega^\omega$ we write $\mathbb{N}^\omega$, hence it is enough to prove the following:
$$
\{\mathbb{N}^\omega\}\text{-universally-Baire} \Rightarrow \{\kappa^\omega \mid 2 \leq \kappa < \omega\}\textrm{-universally Baire}. 
$$
For any cardinal $\kappa$ such that $2 \leq \kappa < \omega$ the space $\kappa^\omega$ is homeomorphic to the Cantor set and by Claim \ref{thm: polish=omega} we have that $\n$ implies $\p$, which finishes the proof of $\ref{8} \Rightarrow \ref{7}$.

The implication $\ref{7} \Rightarrow \ref{4}$ is simply \cite[1K Theorem]{Fremlin}. 

Finally, we prove $\ref{4} \Rightarrow \ref{1}:$ First assume that $X$ is compact. In this case, by \cite[1E Theorem]{Fremlin} $\cech$ implies $\topsp$.

Let us turn to the general case, where we do not assume that $X$ is compact. Every Polish space has a Polish compactification (that is, a compact Polish space containing the original space as a dense subspace, see \cite[Theorem 4.14.]{Ke}), so let $\tilde{X}$ be a Polish compactification of $X$. 

Next we prove that the set $A$ is $\cech$ in the space $\tilde{X}$. Take an arbitrary \v{C}ech-complete space $Y$ and a continuous $f:Y \longrightarrow \tilde{X}$, we need to prove that $f^{-1}(A)$ has the property of Baire in $Y$. First notice that $X$ is a $G_\delta$ subspace of $\tilde{X}$, because it is a Polish subspace of a Polish space, so $f^{-1}(X)$ is a $G_\delta$ subspace of $Y$, hence $f^{-1}(X)$ is a \v{C}ech-complete space. If we denote by $g$ the restriction of $f$ to $f^{-1}(X)$, then $g$ is a continuous function from the \v{C}ech-complete space $f^{-1}(X)$ to $X$, hence $g^{-1}(A)=f^{-1}(A)$ has the property of Baire in $f^{-1}(X)$. But $f^{-1}(X)$ has the property of Baire in $Y$ and by Lemma \ref{bairetultranz} we get that $f^{-1}(A)$ has the property of Baire in $Y$. Therefore $A$ is $\cech$ in the space $\tilde{X}$.

By the compact case of the statement we obtain that $A$ is topological-universally Baire in $\tilde{X}$ and then Lemma $\ref{bővebbenuniveakkorkisebbenis}$ yields that $A$ is topological-universally Baire in $X$.
\ep

Let us continue with the proof of the implications in the left frame.

\begin{claim}
\label{thm: polish=omega} For a subset $A \s X$, the following equivalence holds:
$$\n \Longleftrightarrow \p.$$
\end{claim}

\bp
It is clear that if the set $A$ is $\p$, then it implies that $A$ is also $\n$, so we need to prove the other direction. 

Let $A$ be $\n$, our goal is to prove that $A$ is $\p$. For this, let $Y$ be an arbitrary Polish space and $f:Y \longrightarrow X$ be an arbitrary continuous function. We need to prove that $f^{-1}(A)$ has the property of Baire. 

Take a partition of $Y$ in a form $Y= Y_\textsl{perf} \cup  Y_\textsl{ctbl}$, where $Y_\textsl{perf}$ and $Y_\textsl{ctbl}$ are disjoint and $Y_\textsl{perf}$ is a perfect subset and $Y_\textsl{ctbl}$ is a countable open set. If $Y$ is countable, then each of its subsets have the property of Baire. Hence, we can assume that  $Y$ is uncountable. In this case the above $Y_\textsl{perf}$ is nonempty.

It is enough to prove that $f^{-1}(A) \cap Y_\textsl{perf}$ has the property of Baire in $Y$, because $f^{-1}(A) \cap Y_\textsl{ctbl}$ is countable. By Lemma \ref{bairetultranz} it is enough to prove that $f^{-1}(A) \cap Y_\textsl{perf}$ has the property of Baire in $Y_\textsl{perf}$.

All in all, if $g$ denotes the restriction of $f$ to $Y_\textsl{perf}$, it is enough to prove that $g^{-1}(A)$ has the property of Baire in $Y_\textsl{perf}$. Let $N$ be a dense subspace of $Y_\textsl{perf}$ homeomorphic to $\bs$, by Lemma \ref{perfectcontbairelemma} such a subspace exists. The set $Y_\textsl{perf} \setminus N$ is meager, because $N$ is Polish, hence it is a $G_\delta$ subspace of $Y_\textsl{perf}$, which is also dense, so it is enough to prove that $g^{-1}(A) \cap N$ has the property of Baire in $Y_\textsl{perf}$. Again, by Lemma \ref{bairetultranz}, it is enough to prove that $g^{-1}(A) \cap N$ has the property of Baire in $N$. Denoting by $h$ the restriction of $g$ to $N$, we only need to prove that $h^{-1}(A)=g^{-1}(A) \cap N$ has the property of Baire in $N$.  But the latter is trivial, because $h$ is also a continuous function from the space $N$ homeomorphic to $\mathbb{N}^\omega$, hence $h^{-1}(A)$ has the property of Baire in $N$.
\ep

We prove now that the notions in the right frame are equivalent:

\begin{claim}\label{right box notions are the same}
For a subset $A \s X$, the following equivalence holds:
$$\ch  \Longleftrightarrow \compact.$$
\end{claim}

\begin{proof}
It is clear that if the set $A$ is $\compact$, then $A$ is also $\ch$, so we need to prove the other direction.

Suppose that we are given a set $A$ which is compact Hausdorff-universally Baire and our goal is to prove that $A$ is compact-universally Baire.

Take an arbitrary compact space $Y$ and a continuous $f:Y \longrightarrow X$, we need to show that $f^{-1}(A)$ has the property of Baire in $Y$. The image $f(Y)$ is compact, hence $f(Y) \cap A$ is a $\ch$ set in $X$. By Lemma $\ref{bővebbenuniveakkorkisebbenis}$ we have that $f(Y) \cap A$ is $\ch$ in the subspace $f(Y)$. Theorem $\ref{thm:diagram2}$ (which we prove in the next section) guarantees that the set $f(Y) \cap A$ is $\compact$ in $f(Y)$, hence $f^{-1}(f(Y) \cap A)=f^{-1}(A)$ has the property of Baire in $Y$.
\end{proof}

Finally, we will prove that the notions in the bottom frame are equivalent:

\begin{claim}\label{bottom box notions are the same}
For a subset $A \s X$, the following equivalence holds:
$$
\cantor \Longleftrightarrow \cm.
$$
\end{claim}

\bp

It is clear that if the set $A$ is $\cm$, then $A$ is also $\cantor$, so we need to prove the other direction. 

Let $A$ be $\{2^\omega\}$-universally Baire, our goal is to prove that $A$ is $\cm$.

Recall that a continuous surjective function $f:Y \longrightarrow Z$ between topological spaces is called \emph{irreducible} if for every proper closed subset $F \s Y$ the image $f(F)$ is a proper closed subset of $Z$.
In \cite[Lemma 2.1.]{Woods} it was proved that if $K$ is a perfect compact metric space, then there exists an irreducible mapping $f:2^\omega \longrightarrow K$.
By \cite[pg. 455]{PW} if $f:Y \longrightarrow Z$ is an irreducible mapping, then for every nowhere dense closed set $N \subseteq Y$ we have that $f(N)$ is nowhere dense in $Z$, therefore if $M \subseteq Y$ is meager then $f(M)$ is meager in $Z$.

Take an arbitrary compact metric space $K$  and an arbitrary continuous $g:K \longrightarrow X$, we need to check that $g^{-1}(A)$ has the property of Baire in $K$.

If $K$ is countable then the statement is trivial, so we can assume that $K$ is uncountable. As above, let $K=K_\textsl{perf} \cup K_\textsl{ctbl}$, then $K_\textsl{perf}$ is uncountable.

Since $K_\textsl{ctbl}$ is countable, it is enough to prove that $g^{-1}(A) \cap K_\textsl{perf}$ has the property of Baire in $K$. By Lemma \ref{bairetultranz}, it is enough to prove that $g^{-1}(A) \cap K_\textsl{perf}$ has the property of Baire in $K_\textsl{perf}$. We denote by $\tilde{g}$ the restriction of $g$ to $K_\textsl{perf}$. Choose an arbitrary irreducible map $f:2^\omega \to K_\textsl{perf}$, since $ \tilde{g} \circ f$ is continuous, by assumption, $f^{-1}(\tilde{g}^{-1}(A))$ has the property of Baire in $2^\omega$, so we can write $f^{-1}(\tilde{g}^{-1}(A))=B \cup  M$, where $B$ is Borel and $M$ is meager. Since $f$ was also surjective, we have that $g^{-1}(A) \cap K_\textsl{perf}=\tilde{g}^{-1}(A)=f(f^{-1}(\tilde{g}^{-1}(A))) = f(B) \cup f(M)$. We have that $f(B)$ is an analytic set (see \cite[Theorem 14.2.]{Ke}), so it has the property of Baire (see \cite[Theorem 21.6.]{Ke}) and $f(M)$ is meager by the irreducibility of the function $f$, so $f(B) \cup f(M)$ has the property of Baire in $K_\textsl{perf}$.
\ep

With the proofs of Claim \ref{top box notions are the same} - \ref{bottom box notions are the same} the proof of Theorem \ref{thm:diagram1}. is complete.

\subsection{Consistent negative results}

Now we prove that there are no more $ZFC$ implications in Theorem \ref{thm:diagram1}.

\begin{theorem}
\label{t:noneq}
Under the Continuum Hypothesis the four notions of universal Baireness (topological-universally Baire, Polish-universally Baire, compact-universally Baire and compact metrizable-universally Baire) are pairwise non\-equivalent (in suitably chosen Polish spaces).
\end{theorem}

\begin{proof}
It is easy to check that we only need to show that there exists a set $S$ which is $\p$ but not $\compact$ and there exists a set $L$ which is $\compact$ but not Polish-universally Baire (in suitably chosen Polish spaces). Therefore the following two examples will complete the proof.

\begin{ex}
\label{ex:S}
Let $S\s [0,1]$ be a Sierpi\'nski set. Then $S$ is Polish-universally Baire $(\Longleftrightarrow \cm$ by Theorem \ref{thm:diagram2}), but not $\compact$ $( \Longleftrightarrow \topsp$ by Theorem \ref{thm:diagram2}).
\end{ex}

\bp
 First  we prove that $S$ is $\p$. 
Let $Y$ be a Polish space and
$f: Y \longrightarrow [0,1]$ be continuous, we need to show that $f^{-1}(S)$ has the property of Baire. We can find a dense $G_\delta$ null set $G \s [0,1]$ such that  $f^{-1}(G)$ is dense (take the image of a countable dense set and complete it to a countable dense subset of $[0,1]$, then cover this set by a $G_\delta$ null set).
Then $S \cap G$ is countable, hence $f^{-1}(S \cap G)$ is $F_\sigma$, so has the property of Baire. On the other hand, $f^{-1}(S \cap G^c)$ is disjoint from the dense $G_\delta$ set $f^{-1}(G)$, hence it is meager and therefore has the property of Baire. Thus, $f^{-1}(S) = f^{-1}(S \cap G) \cup f^{-1}(S \cap G^c)$ also has the property of Baire.

Next, it is well-known and easy to see that a Sierpi\'nski set is not measurable. By Theorem \ref{thm:diagram2}, if a set $D$ is compact-universally Baire in $[0,1]$ then $D$ is \v{C}ech-complete-universally Baire, and by \cite[1C Proposition]{Fremlin} any $\cech$ set is Lebesgue measurable. Hence $S$ is not $\compact$.
\ep

\begin{ex}
\label{ex:L}
Let $L \s \bs$ be a Luzin set. Then $L$ is compact-universally Baire but not $\p$.
\end{ex}

\bp
First we show that $L$ is $\compact$. Fix a compact space $K$ and a continuous function $f:K \longrightarrow \bs$. Since the set $f(K)$ is a compact subset of $\bs$, it is also meager, so $f(K)\cap L$ is countable, hence $f^{-1}(L)$ is $F_\sigma$, so has the property of Baire in $K$.

Finally, since it is well-known and easy to see that a Luzin set cannot have the property of Baire, the identity function witnesses that $L$ is not $\p$.
\ep

This concludes the proof of Theorem \ref{t:noneq}.
\end{proof}

Let us now mention few more consistent examples that use hypotheses other than $CH$. 

Recall that $\add(\mathcal{M})$ denotes the \emph{additivity} of the meager sets of $\bs$, that is, the least cardinal $\kappa$, such that there exist meager sets $\{M_\alpha \mid \alpha < \kappa \}$ in $\bs$ such that $\underset{\alpha <\kappa} {\cup} M_\alpha$ is not meager. Note that replacing the space $\bs$ by any perfect Polish space, the cardinality $\add(\mathcal{M})$ does not change by Lemma \ref{perfectcontbairelemma}. It is well-known that the union of less than $\add(\mathcal{M})$ many sets with the property of Baire still has the property of Baire.

We also recall the \emph{cofinality} of Lebesgue null sets of $[0,1]$, which is denoted by $\cof(\mathcal{N})$. It is the least cardinal $\kappa$, such that there exist Lebesgue null sets $\{N_\alpha \mid \alpha < \kappa \}$ in $[0,1]$, which satisfy that for every Lebesgue null set $N \s [0,1]$ there exists an $\alpha$ such that $N \s N_\alpha$. The \emph{covering number} of the Lebesgue null sets is defined as the least cardinal $\kappa$ such that there exists Lebesgue null sets $\{N_\alpha \mid \alpha < \kappa \}$ in $[0,1]$ for which $[0,1] = \underset{\alpha < \kappa}{\cup} N_\alpha$. We denote the covering number by $\cov(\mathcal{N})$. Recall that there are various models of set theory where $\add(\mathcal{M})=\cov(\mathcal{N})=\cof(\mathcal{N})$, for example, under Martin's Axiom each of the previous cardinals equals $2^\omega$.

\begin{ex}\label{ex: modified Sieprinnski}
Suppose that $\add(\mathcal{M})=\cov(\mathcal{N})=\cof(\mathcal{N})$, then there exists $A \s [0, 1]$ that is Polish-universally Baire ($\iff$ compact metrizable-universally Baire by Theorem \ref{thm:diagram2}) but not compact-universally Baire ($\iff$ topological-universally Baire by Theorem \ref{thm:diagram2}).
\end{ex}

\begin{proof}
    Let $\kappa=\cov(\mathcal{N})$ and using the cardinal assumption $\cov(\mathcal{N}) = \cof(\mathcal{N})$, one can easily build by transfinite recursion a "Sierpi\'nski-like" set $A$ of cardinality $\kappa$ such that $\abs{A \cap N} < \kappa$, whenever $N$ is a Lebesgue null set in $[0,1]$.
    
    Note that such a set $A$ cannot be Lebesgue measurable. It was shown in \cite[1C Proposition]{Fremlin} that if a set $D$ is topological-universally Baire $(\Leftrightarrow \cech$), then $D$ is Lebesgue measurable, hence $A$ is not topological-universally Baire.

    We prove now that $A$ is $\p$. Take an arbitrary Polish space $Y$ and a continuous $f:Y \longrightarrow X$, we need to prove that $f^{-1}(A)$ has the property of Baire. Choose an arbitrary $G_\delta$ Lebesgue null set $G \s [0,1]$ such that $f^{-1}(G)$ is dense in $Y$ (as we did in Example \ref{ex:S}). We have that $\abs{A \cap G} < \add(\mathcal{M})$, hence $f^{-1}(A \cap G)$ is the union of less than $\add(\mathcal{M})$ many $F_\sigma$ sets, so $f^{-1}(A \cap G)$ has the property of Baire. On the other hand, the set $f^{-1}(A \cap G^c) \s f^{-1}(G^c)$, and the latter set is clearly meager in $Y$, which finishes the proof that the set $A$ is $\p$.   
\end{proof}

To give more consistent examples, we recall the cardinal $\add (\mathcal{N})$, the \emph{additivity} of the Lebesgue null sets of $[0, 1]$, that is, the least cardinal $\kappa$, such that there exist Lebesgue null sets $\{N_\alpha \mid \alpha < \kappa \}$ in $[0, 1]$, such that $\underset{\alpha < \kappa}{\cup} N_\alpha$ is not a Lebesgue null set. Recall that there exists a model of set theory where $\add (\mathcal{N}) < \add(\mathcal{M})$, see for example \cite[Model 7.6.9.]{Bart} where $\omega_1 = \add (\mathcal{N}) < \add (\mathcal{M})= \omega_2 = 2^\omega$ holds.

\begin{ex}
\label{ex:3}
Suppose that $\add (\mathcal{N}) < \add (\mathcal{M})$, then there exists $A \s [0, 1]$ that is Polish-universally Baire ($\iff$compact metrizable-universally Baire by Theorem \ref{thm:diagram2}) but not topological-universally Baire ($\iff$compact-universally Baire by Theorem \ref{thm:diagram2}).
\end{ex}

\bp
Let $\lambda$ denote the Lebesgue measure and let $\kappa=\add(\mathcal{N})$. Then there exists $ \{ N_\alpha \mid \alpha < \kappa \}$ in $[0, 1]$ with $\lambda(N_\alpha)=0$ for every $\alpha < \kappa$ such that 
$$
N:=\bigcup\limits_{\alpha < \kappa} N_\alpha
$$ 
has positive outer measure. We cover the set $N_\alpha$ by a Borel null set $\tilde{B}_\alpha$. We will make the $\tilde{B}_\alpha$-s disjoint by setting $B_\alpha:=\tilde{B}_\alpha \setminus \bigcup\limits_{\beta < \alpha} \tilde{B}_\beta$. Note that for every $\alpha < \kappa$ the set $B_\alpha$ is a null set. By the Four Poles Theorem \cite{4P} there exists $E \s \kappa$, such that $\bigcup\limits_{\alpha \in E} B_\alpha$ is not measurable. Set $A:=\bigcup\limits_{\alpha \in E} B_\alpha$. 

We prove now that the set $A$ is $\p$. Take an arbitrary Polish space $Y$ and a continuous $f:Y \longrightarrow [0,1]$, we need to check that $f^{-1}(A)$ has the property of Baire. Since the union of less than $\add(\mathcal{M})$ many sets with the property of Baire in a Polish space also has the property of Baire, we obtain that $f^{-1}\big(\bigcup\limits_{\beta < \alpha} \tilde{B}_{\beta}\big)=\bigcup\limits_{\beta < \alpha} f^{-1}(\tilde{B}_{\beta})$ has the property of Baire for every $\alpha <\kappa$. The latter also implies that $f^{-1}(B_{\alpha})= f^{-1}(\tilde{B}_{\alpha}) \setminus \bigcup\limits_{\beta < \alpha} f^{-1}(\tilde{B}_{\beta})$ has the property of Baire. Since $\kappa < \add (\mathcal{M})$ we obtain that $f^{-1}(A) = \bigcup\limits_{\alpha \in E} f^{-1}(B_\alpha)$ has the property of Baire.

On the other hand, as above, it was shown in \cite[1C Proposition]{Fremlin} that if a set $D$ is topological-universally Baire $(\Leftrightarrow \cech$), then $D$ is Lebesgue measurable, hence $A$ is not topological-universally Baire.
\ep

Let us give another consistent example, using uniformity numbers. The \emph{uniformity number} of the Lebesgue null sets of $[0,1]$, denoted by $\non(\mathcal{N})$, is the least cardinal $\kappa$ such that there exists a set $A \s [0,1]$ with cardinality $\kappa$, such that $A$ is not Lebesgue null. Similarly, the \emph{uniformity number} of the meager sets of $\bs$, denoted by $\non(\mathcal{M})$, is the least cardinal $\kappa$ such that there exists a set $A \s \bs$ with cardinality $\kappa$, such that $A$ is not a meager set. Note that replacing $\bs$ by any perfect Polish space in the definition of $\non(\mathcal{M})$, we get the same cardinality by Lemma \ref{perfectcontbairelemma}.

We also recall the definition of \emph{universally meager} sets. A set $M$ in a perfect Polish space $Y$ is called universally meager, if for every Borel isomorphism $f:Y \longrightarrow Y$, the set $f^{-1}(M)$ is meager in $Y$. There are many nontrivial equivalent formulations given by Zakrzewski in \cite{Zak1}, \cite{Zak2} when a set $M$ is universally meager. The crucial property we will use is that universally meager sets are $\p$. This will follow later as follows. By \cite[Theorem 2.1, (iii)]{Zak1}, a set $M \s Y$ is universally meager if and only if for every $\sigma$-ideal $\mathcal{I}$ on $\mathcal{B}(Y)$ such that $\mathcal{B}(Y) /\mathcal{I} \simeq \cat$ there exists a Borel set $B \in \mathcal{I}$ such that $M \s B$ (here $\cat$ denotes the category algebra, for a definition and further properties, see Section \ref{Boolean section}). If the previous assumption holds for $M$, then $M$ obviously satisfies \ref{B(X) modulo I = C} in Theorem \ref{thm: boole equivalence} below, which is equivalent to $M$ being $\p$.

\begin{ex}\label{ex:uniformity number}
    Suppose that $\non(\mathcal{N}) < \non(\mathcal{M})$, then there exists $A \s [0, 1]$ that is Polish-universally Baire ($\iff$ compact metrizable-universally Baire by Theorem \ref{thm:diagram2}) but not topological-universally Baire ($\iff$ compact-universally Baire by Theorem \ref{thm:diagram2}).
\end{ex}

\begin{proof}
    Take a set $M \s [0,1]$ of size $\non(\mathcal{N})$ which is not Lebesgue measurable. As in some of the previous examples, by \cite[1C Proposition]{Fremlin} a Lebesgue nonmeasurable set cannot be $\topsp$.

    On the other hand, it is clear that any subset of $[0,1]$ with cardinality strictly less than $\non(\mathcal{M})$ must be universally meager (a Borel isomorphism does not change the cardinality of a set), hence $M$ is universally meager, which implies that $M$ is $\p$.
\end{proof}

In the final example we assume the existence of some large cardinals.

\begin{ex}\label{ex:twowoodincardinal}
    Suppose that there exist two Woodin cardinals, then there exists $A \s [0, 1]$ that is Polish-universally Baire ($\iff$ compact metrizable-universally Baire by Theorem \ref{thm:diagram2}) but not topological-universally Baire ($\iff$ compact-universally Baire by Theorem \ref{thm:diagram2}).
\end{ex}

\begin{proof}
    Let $H$ be a Hausdorff gap (for more, see \cite[Theorem 29.7]{Jech}). Formally, a Hausdorff gap is a special subset $H \s 2^\omega$ (or sometimes it is a subset of $\bs$) of size $\omega_1$. By \cite[Proposition 2.10. (b)]{Zak3}, a Hausdorff gap is countably perfectly meager and by \cite[Theorem 7, $(3) \Rightarrow (2) \Rightarrow (1)$]{Bartperf} every countably perfectly meager set is universally meager, hence a Hausdorff gap is universally meager. By a Borel isomorphism $f$ between $2^\omega$ and $[0,1]$ we obtain the set $A = f(H)$ of size $\omega_1$ where $A$ is universally meager, hence $A$ is $\p$.

    Suppose for a contradiction that $A$ is topological-universally Baire. The existence of two Woodin cardinals implies that every $\topsp$ set is determined by \cite[Theorem 5.4.]{FMW}. Now we consider the $\boldsymbol{^*}$-$\boldsymbol{game}$ for $A$ (for the definition and relevant results see \cite[Section 21.A]{Ke}), the winning set for \textit{I} is a continuous preimage of $A$ in some Polish space, so the winning set for \textit{I} is determined, therefore either \textit{I} or \textit{II} has a winning strategy in this game, hence $A$ has the perfect set property. Since $A$  is not countable nor contains a perfect set, we get a contradiction.
\end{proof}

\section{Universally Baire sets in \texorpdfstring{$\sigma$}--compact Polish spaces}\label{subsectionunivbaire2}

In this section we will prove that when $X$ is a $\sigma$-compact Polish space, the diagram simplifies.

\bt
\label{thm:diagram2}
For a subset of a $\sigma$-compact Polish space, the following implications hold:

\begin{center}
    \begin{tikzcd}[row sep = 0cm, column sep = 2cm]
     \fbox{
\begin{tabular}{c} topological-universally Baire \\
$\Updownarrow$
\\
Hausdorff-universally Baire \\
$\Updownarrow$
\\
$\{Y \mid Y \text{ has a regular open basis} \}$-universally Baire\\
$\Updownarrow$
\\ \text{ \v{C}ech-complete-universally Baire} \\
$\Updownarrow$
\\
metrizable-universally Baire\\
$\Updownarrow$
\\
completely metrizable-universally Baire\\
$\Updownarrow$

\\ $\{ \lambda^\omega \mid \lambda \text{ is a cardinal} \}$-universally Baire \\
$\Updownarrow$
\\ $\{ \lambda^\omega \mid \lambda \text{ is an infinite cardinal} \}$-universally Baire \\
$\Updownarrow$
\\ $\{ Y \mid Y \text{ is an ordered space} \}$-universally Baire \\
$\Updownarrow$
\\ compact-universally Baire\\
$\Updownarrow$
\\ compact Hausdorff-universally Baire \end{tabular}}
 \\
\Bigg\Downarrow
\\
\fbox{\begin{tabular}{c} Polish-universally Baire \\
$\Updownarrow$
\\ $\{ \mathbb{N^\omega} \}$-universally Baire \\
$\Updownarrow$
\\  compact metrizable-universally Baire \\
$\Updownarrow$
\\ $\{ 2^\omega \}$-universally Baire \end{tabular}}
    \end{tikzcd}
\end{center}

\smallskip
Moreover, under the Continuum Hypothesis the two notions in the above diagram are nonequivalent.
\et

We will prove Theorem \ref{thm:diagram2} in two steps. First, we will assume that the space $X$ is also compact, after that we will move towards the $\sigma$-compact case. The following claims and their proofs establish Theorem \ref{thm:diagram2}. We start with the top frame:

\begin{claim}\label{kompaktra ekvivalens a fremlim és komphaus}
Let $X$ be a compact metric space. For a subset $A \s X$, the following equivalence holds:
$$\ch \Longleftrightarrow \topsp.$$
\end{claim}

\bp
It is clear that if the set $A$ is $\topsp$, then $A$ is also $\ch$. The other direction of the implication follows from \cite[1E Theorem]{Fremlin}.
\ep

We will now weaken the compactness assumption:

\begin{claim}\label{szigmakompaktra ekvivalencia fremlinre}
Let $X$ be a $\sigma$-compact Polish space. For a subset $A \s X$, the following equivalence holds:
$$
\ch \Longleftrightarrow \topsp.
$$
\end{claim}

\bp
It is clear that if the set $A$ is $\topsp$, then $A$ is also $\ch$, so we need to prove the other direction.

Suppose that $A$ is $\ch$ and our goal is to derive that $A$ is $\topsp$. Let $Y$ be a topological space and take an arbitrary continuous $f:Y \longrightarrow X$, we need to check that $f^{-1}(A)$ has the property of Baire in $Y$.

Let $X=\bigcup\limits_{n \in \omega} C_n$, where all the $C_n$-s are compact subspaces of $X$. We can assume that all the $C_n$-s are nonempty. First note that by Lemma \ref{bővebbenuniveakkorkisebbenis} the sets
$C_n \cap A$ are \ch \; in $C_n$, hence by Claim \ref{kompaktra ekvivalens a fremlim és komphaus} we have that $C_n \cap A$ is \topsp \; in $C_n$. 

Fix $n$ and let $f_n$ denote the restriction of $f$ to $f^{-1}(C_n)$, then  $f_n^{-1}(C_n \cap A) = f^{-1}(C_n) \cap f^{-1}(A)$ has the property of Baire in $f^{-1}(C_n)$.  Since the set $C_n$ is a closed subspace of $X$, the set $f^{-1}(C_n)$ is a closed subspace of $Y$, hence $f^{-1}(C_n)$ has the property of Baire in $Y$ and by Lemma \ref{bairetultranz} the set $f^{-1}(C_n) \cap f^{-1}(A)$ has the property of Baire in $Y$. The sets with the property of Baire form a $\sigma$-algebra, therefore $f^{-1}(A) = \bigcup\limits_{n \in \omega} (f^{-1}(C_n) \cap f^{-1}(A))$ has the property of Baire in $Y$.
\ep

To see that there is no implication from the bottom frame to the top frame in ZFC, that is, $\p$ $\Rightarrow$ $\topsp$ is not provable in $ZFC$, even when the space $X$ is compact, we refer to Example \ref{ex:3}. 

Finally, we will prove that the notions in the bottom frame of the diagram in Theorem \ref{thm:diagram2} are equivalent, we start with the case where $X$ is compact metric:

\begin{claim}\label{thm:kompaktban CM=P}
Let $X$ be a compact metric space. For a subset $A \s X$, the following equivalence holds:
$$\p \Longleftrightarrow \cm.$$
\end{claim}

\begin{proof} It is clear that if the set $A$ is $\p$, then $A$ is compact metrizable-universally Baire, so we need to prove the other direction.

Suppose that the set $A$ is $\cm$ and our goal is to prove that $A$ is $\p$. Take an arbitrary Polish space $Y$  and an arbitrary continuous $f:Y \longrightarrow X$, we need to check that $f^{-1}(A)$ has the property of Baire in $Y$. 

First, we identify $X$ with a compact (equivalently a closed) subspace of the Hilbert cube $[0,1]^\omega$ (this is possible by a well-known result, see \cite[Theorem 4.14.]{Ke}). Then the function $f$ has the form $f=(f_n)_{n \in \omega}$, where $f_n$ is a continuous function $Y \to [0,1]$ for every $n$.

We choose a metric $d$ on $Y$ compatible with its topology with $d \leq 1$. Let $(y_n)_{n \in \omega}$ be dense in $Y$ and let $g_n$ denote the continuous function 
$$g_n:Y \to [0,1], \; g_n(y):=d(y,y_n).$$

It is easy to see that the $g_n$-s separate the points, that is $\forall y \neq y'\in Y \; \exists n$ such that $g_n(y) \neq g_n(y')$. We consider the following embedding of the space $Y$ into the Hilbert cube $\imath: Y \to [0,1]^\omega$:
$$
\imath: \; y \mapsto (f_0(y), g_0(y), f_1(y),g_1(y), \dots ).
$$
It is easy to see that $\imath$ is really an embedding (or see \cite[Theorem 4.14.]{Ke}), so from now on we identify $Y$ with $\imath(Y)$. 

After the previous identifications, we can see that $f$ takes the point  $(f_0(y), g_0(y), f_1(y),g_1(y), \dots )$ to $(f_0(y),f_1(y), \dots )$, in other words, $f$ is simply a projection on the even coordinates. So we can extend $f$ to a $\tilde{f}:[0,1]^\omega \to [0,1]^\omega$ by the formula
$$
\tilde{f}: \; (x_0,y_0,x_1, y_1, \dots ) \mapsto (x_0,x_1, \dots )
$$
which is clearly continuous. By the continuity of $\tilde{f}$ we have that $\tilde{f}(\overline{Y}) \subseteq \overline{f(Y)} \subseteq X$, where for the latter containment we used the fact that $X$ is closed in $[0,1]^\omega$.

We will denote by $h$ the restriction of the function $\tilde{f}$ to $\overline{Y}$. Since $\overline{Y}$ is compact and $h: \overline{Y} \longrightarrow X$ is continuous and $A$ is $\cm$, we have that $h^{-1}(A)$ has the property of Baire in $\overline{Y}$ and by the fact that $Y$ is comeager in $\overline{Y}$ we have that $h^{-1}(A) \cap Y = f^{-1}(A)$ has the property of Baire in $Y$, which finishes the proof.
\end{proof}

From compact Polish spaces we can easily go to $\sigma$-compact Polish spaces:

\begin{claim}\label{claim:szigmakompaktrajobboldal ekvi}
Let $X$ be a $\sigma$-compact Polish space. For a subset $A \s X$, the following equivalence holds:
$$\p \Longleftrightarrow \cm.$$
\end{claim}

\begin{proof}
It is clear that if the set $A$ is $\p$, then $A$ is compact metrizable-universally Baire, so we need to prove the other direction.

Suppose that the set $A$ is $\cm$ and our goal is to prove that $A$ is $\p$. Take an arbitrary Polish space $Y$  and an arbitrary continuous $f:Y \longrightarrow X$, we need to check that $f^{-1}(A)$ has the property of Baire in $Y$.

Let $X=\bigcup\limits_{n \in \omega} C_n$, where all the $C_n$-s are compact subspaces of $X$ (we can assume that every $C_n$ is nonempty). By Lemma \ref{bővebbenuniveakkorkisebbenis} we have that $ C_n \cap A$ is a compact metrizable-universally Baire set in $C_n$ for every $n$, so by Claim \ref{thm:kompaktban CM=P} the set $\; C_n \cap A$ is also Polish-universally Baire in $C_n$. Let $f_n$ denote the restriction of $f$ to $f^{-1}(C_n)$. Since $f^{-1}(C_n)$ is a closed set in $Y$, it is a Polish subspace. Since $C_n \cap A$ is Polish-universally Baire, $f_n^{-1}(C_n \cap A)=f^{-1}(C_n) \cap f^{-1}(A)$ has the property of Baire in $f^{-1}(C_n)$, hence by Lemma \ref{bairetultranz} the set $f^{-1}(C_n) \cap f^{-1}(A)$ has the property of Baire in $Y$, so $f^{-1}(A)= \bigcup\limits_{n \in \omega} (f^{-1}(C_n) \cap f^{-1}(A))$ has the property of Baire in $Y$.
\end{proof}

With the proofs of Claim \ref{kompaktra ekvivalens a fremlim és komphaus} - \ref{claim:szigmakompaktrajobboldal ekvi} the proof of Theorem \ref{thm:diagram2} is complete.

\customsection{Definitions using Borel and Baire measurable preimages}{Definitions using Borel and Baire measurable preimages}{Borel and Baire measurable preimages}

In this section we replace continuous preimages in the definition of $\mathcal{F}$-universally Baire by Borel and Baire measurable preimages. Besides being natural generalizations, we would like to point out that historically the first definition by Christensen \cite{Chr} also used Borel preimages.
\bigskip

Let us start with the Borel version of Definition \ref{def:univBaire}.

\begin{definition}
Let $\mathcal{F}$ be a class of topological spaces and $A \subseteq X$. The set $A$ is called \emph{$\mathcal{F}$-Borel-universally Baire}, if $f^{-1}(A)$ has the property of Baire for every $Y \in \mathcal{F}$ and for every Borel function $f: Y \longrightarrow X$.
\end{definition}

\begin{definition}
Let $\mathcal{F}$ be a class of topological spaces and $A \subseteq X$. The set $A$ is called \emph{$\mathcal{F}$-Baire-universally Baire}, if $f^{-1}(A)$ has the property of Baire for every $Y \in \mathcal{F}$ and for every Baire measurable function $f: Y \longrightarrow X$.
\end{definition}

The following easy lemma will show us that the two definitions are equivalent:

\begin{lemma}\label{lemma:Borel=Baire}
    For a subset $A \s X$, the following equivalence holds:
    $$
    \mathcal{F}\text{-Borel-universally Baire} \Longleftrightarrow \mathcal{F}\text{-Baire-universally Baire}.
    $$
\end{lemma}

\begin{proof}
    It is clear that if $A$ is $\mathcal{F}\text{-Baire-universally Baire}$ then $A$ is also $\mathcal{F}$-Borel-universally Baire. To see the converse direction, we assume that $A$ is $\mathcal{F}\text{-Borel-universally Baire}$ and we consider a space $Y \in \mathcal{F}$ and an arbitrary Baire measurable function $f:Y \longrightarrow X$. Since $X$ is second countable, there exists a comeager subset $\tilde{Y} \s Y$ and a Borel function $g:Y \longrightarrow X$ such that $f\vert_{\tilde{Y}}=g\vert_{\tilde{Y}}$. By assumption, the set $g^{-1}(A)$ has the property of Baire in $Y$ and $f^{-1}(A) \vartriangle g^{-1}(A) \s Y \setminus \tilde{Y}$, hence we conclude that the set $f^{-1}(A)$ has the property of Baire.
\end{proof}

As opposed to the case of continuous functions, for Borel/Baire functions compactness will no longer play any role, hence the four notions will collapse to two.

\begin{theorem}\label{Borel preimages theorem}
For a subset of a Polish space, the following implications hold:

\begin{center}
    \begin{tikzcd}[row sep = 0cm, column sep = 2cm]
     \fbox{
\begin{tabular}{c}
topological-universally Baire \\
$\Updownarrow$
\\
topological-Borel/Baire-universally Baire \\
$\Updownarrow$
\\
Hausdorff-Borel/Baire-universally Baire \\
$\Updownarrow$
\\
$\{Y \mid Y \text{ has a regular open basis} \}$-Borel/Baire-universally Baire\\
$\Updownarrow$
\\ \text{ \v{C}ech-complete-Borel/Baire-universally Baire} \\
$\Updownarrow$
\\
metrizable-Borel/Baire-universally Baire\\
$\Updownarrow$
\\
completely metrizable-Borel/Baire-universally Baire\\
$\Updownarrow$

\\ $\{ \lambda^\omega \mid \lambda \text{ is a cardinal} \}$-Borel/Baire-universally Baire \\
$\Updownarrow$
\\ $\{ \lambda^\omega \mid \lambda \text{ is an infinite cardinal} \}$-Borel/Baire-universally Baire \\
$\Updownarrow$
\\ $\{ Y \mid Y \text{ is an ordered space} \}$-Borel/Baire-universally Baire \\
$\Updownarrow$
\\ compact-Borel/Baire-universally Baire\\
$\Updownarrow$
\\ compact Hausdorff-Borel/Baire-universally Baire
\end{tabular}}
 \\
\Bigg\Downarrow
\\
\fbox{\begin{tabular}{c}
Polish-universally Baire \\
$\Updownarrow$
\\ 
Polish-Borel/Baire-universally Baire \\
$\Updownarrow$
\\ $\{ \mathbb{N^\omega} \}$-Borel/Baire-universally Baire \\
$\Updownarrow$
\\  compact metrizable-Borel/Baire-universally Baire \\
$\Updownarrow$
\\ $\{ 2^\omega \}$-Borel/Baire-universally Baire
\end{tabular}}
    \end{tikzcd}
\end{center}

\smallskip
Moreover, consistently the two notions in the above diagram are nonequivalent.
\end{theorem} 

In the following claims we will show that all of the notions in the top frame are equivalent. By Lemma \ref{lemma:Borel=Baire} we know that for a subset $A\s X$ we have that $A$ is $\mathcal{F}\text{-Borel-universally Baire} $ if and only if $ A$ is $\mathcal{F}\text{-Baire-universally Baire}$, hence we will only prove the equivalence of the corresponding $\mathcal{F}\text{-Borel-universally Baire}$ notions.

\begin{claim}\label{topological=topological Borel}
For a subset $A \s X$, the following equivalence holds:
$$ \topsp \Longleftrightarrow \text{topological-Borel-universally Baire}.$$
\end{claim}

\bp
If $A$ is topological-Borel-universally Baire then $A$ is clearly topological-universally Baire.

Conversely, suppose that $A$ is topological-universally Baire and take an arbitrary topological space $Y$ with a Borel function $f:Y \longrightarrow X$, we need to show that $f^{-1}(A)$ has the property of Baire. If $Y$ is meager in itself, then each of its subsets has the property of Baire so we can assume that $Y$ is not meager. By \cite[Theorem 8.38.]{Ke} there is a comeager $G_\delta$ subspace $\tilde{Y}$ such that the restriction of $f$ to $\tilde{Y}$ is continuous, let us denote this restriction by $g$. Then $g^{-1}(A)$ has the property of Baire in $\tilde{Y}$, so by Lemma \ref{bairetultranz} the set $g^{-1}(A)$ has the property of Baire in $Y$. Since $f^{-1}(A)=g^{-1}(A) \cup (f^{-1}(A) \cap (Y \setminus \tilde{Y}))$ and $Y \setminus \tilde{Y}$ is meager, $f^{-1}(A)$ has the property of Baire. 
\ep

\begin{claim}\label{topborel=compactborel}
For a subset $A \s X$, the following 
equivalence holds:
$$ \hspace{-3pt} \text{topological-Borel-universally Baire} \Leftrightarrow \text{compact Hausdorff-Borel-universally Baire}.$$
\end{claim}

\begin{proof}

It is clear that if $A$ is topological-Borel-universally Baire then $A$ is compact Hausdorff-Borel-universally Baire. 

Let us suppose that $A$ is compact Hausdorff-Borel-universally Baire. By Claim \ref{topological=topological Borel} it is enough to prove that $A$ is topological-universally Baire. To prove that $A$ is topological-universally Baire, it is enough to show by Theorem \ref{thm:diagram1} that $A$ is \v{C}ech-complete-universally Baire. Consider an arbitrary \v{C}ech-complete space $Y$ and an arbitrary continuous $f:Y \longrightarrow X$, we need to show that $f^{-1}(A)$ has the property of Baire. We can suppose that $A \neq X$, so let us fix an $a \in X \setminus A$. Let $\tilde{Y}$ be a Hausdorff compactification of $Y$ and define the function

\begin{equation*}
g(y):=
\begin{cases*}
f(y), &  if  $y \in Y,$ \\
a, &  if $y \in \tilde{Y} \setminus Y.$
\end{cases*}
\end{equation*}

Since $Y$ is \v{C}ech-complete, it is a $G_\delta$ subspace of $\tilde{Y}$, hence $g$ is a Borel function. The set $g^{-1}(A)$ has the property of Baire in $\tilde{Y}$ and using that $Y$ is a dense $G_\delta$ subspace of $\tilde{Y}$ we get that $f^{-1}(A)=g^{-1}(A) \cap Y$ has the property of Baire in $Y$.
\end{proof}

Next we prove that the notions in the bottom frame are equivalent.

\begin{claim}\label{cor: a jó választás} 
For a subset $A \s X$, the following equivalence holds:
$$ \text{Polish-universally Baire} \Longleftrightarrow \text{Polish-Borel-universally Baire}.$$
\end{claim}

\bp
If $A$ is Polish-Borel-universally Baire, then $A$ is also Polish-universally Baire. 

Conversely, suppose that $A$ is Polish-universally Baire, and take an arbitrary Polish space $Y$ and a Borel function $f:Y \longrightarrow X$, we need to show that $f^{-1}(A)$ has the property of Baire. The function $f:Y \longrightarrow X$ is Baire measurable, so it is continuous on a dense $G_\delta$ subspace $\tilde{Y} \s Y$ (see \cite[Theorem 8.38.]{Ke}). The subspace $\tilde{Y}$ is Polish, because it is a $G_\delta$ subspace of a Polish space.  If $g$ denotes the restriction of $f$ to $\tilde{Y}$ we get that $g^{-1}(A)$ has the property of Baire in $\tilde{Y}$. It follows from Lemma \ref{bairetultranz} that $g^{-1}(A)$ has the property of Baire in $Y$. As in Claim \ref{topological=topological Borel}, $f^{-1}(A)=g^{-1}(A) \cup (f^{-1}(A) \cap (Y \setminus \tilde{Y}))$ and $Y \setminus \tilde{Y}$ is meager, so $f^{-1}(A)$ has the property of Baire.

\ep

\begin{claim}
For a subset $A \s X$, the following equivalence holds:
$$ \text{Polish-Borel-universally Baire} \Longleftrightarrow \{\mathbb{N}^\omega \}\text{-Borel-universally Baire}.$$
\end{claim}

\begin{proof}
    It is clear that if $A$ is $\text{Polish-Borel-universally Baire}$ then $A$ is also $\{\mathbb{N}^\omega \}\text{-Borel-universally Baire}$.

    Suppose that $A$ is $\{\mathbb{N}^\omega \}\text{-Borel-universally Baire}$. Then $A$ is $\{\mathbb{N}^\omega\}$-universally Baire and by Claim \ref{thm: polish=omega} we have that $A$ is $\p$. Claim \ref{cor: a jó választás} implies that $A$ is {Polish-Borel-universally Baire}.
\end{proof}

Since the implications Polish-Borel-universally Baire $\Longrightarrow$ compact met\-ri\-za\-ble-Borel-universally Baire $ \Longrightarrow$ $\{2^\omega \}$-Borel-universally Baire are trivial, it is enough to prove the following claim to finish the proof that the notions in the bottom frame are equivalent:

\begin{claim}\label{bsub=cantorub}
For a subset $A \s X$, the following equivalence holds:
$$ \{\mathbb{N}^\omega \}\text{-Borel-universally Baire} \Longleftrightarrow \{2^\omega\}\text{-Borel-universally Baire}.$$
\end{claim}

\bp
Suppose that $A$ is $\{\bs\}$-Borel-universally Baire. Then $A$ is $\{\bs\}$-universally Baire. Claim \ref{thm: polish=omega} implies that $A$ is Polish-universally Baire and Claim \ref{cor: a jó választás} implies that $A$ is Polish-Borel-universally Baire, in particular,  $A$ is $\{2^\omega\}$-Borel-universally Baire.

Conversely, suppose that $A$ is $\{2^\omega\}\text{-Borel-universally Baire}$. We will identify $\mathbb{N}^\omega$ with a dense $G_\delta$ subspace $N \s 2^\omega$, which is possible by Lemma \ref{perfectcontbairelemma}. Take an arbitrary Borel function $f:N \longrightarrow X$ and we need to show that $f^{-1}(A)$ has the property of Baire. We can suppose that $ A \neq X$ and fix an $a \in  X \setminus A$. We define the following function $g$ on $2^\omega$ as
\begin{equation*}
    g(y):= 
    \begin{cases*}
    f(y), & $\text{if } y \in N,$ \\
    a, & $\text{if }y \in 2^\omega \setminus N$.
    \end{cases*}
\end{equation*}


Similarly to Claim \ref{topborel=compactborel}, $g$ is a Borel function. The set $g^{-1}(A)$ has the property of Baire in $2^\omega$ and using that $N$ is a dense $G_\delta$ subspace of $2^\omega$ we get that $f^{-1}(A)=g^{-1}(A) \cap N$ has the property of Baire in $N$.
\ep

With the proofs of Claim \ref{topological=topological Borel} - \ref{bsub=cantorub} the proof of Theorem \ref{Borel preimages theorem} is complete. 

\bigskip

Finally, let us finish this section with some simple observations, examining for which $\mathcal{F}$ are the $\mathcal{F}$-universally Baire sets closed under taking Borel preimages. Let us state the following obvious theorem:

\begin{theorem}\label{F-borel-ub borel ose}
    Let $X'$ be a Polish space, $A \s X$ and $f$ be a Borel function $f:X' \longrightarrow X$. If $A$ is $\mathcal{F}$-Borel-universally Baire, then $f^{-1}(A)$ is also $\mathcal{F}$-Borel-universally Baire set in $X'$.
\end{theorem}

\begin{proof}
Consider an arbitrary $Y \in \mathcal{F}$ and a Borel function $g:Y \longrightarrow X'$, we need to show that $g^{-1}(f^{-1}(A))$ has the property of Baire in $Y$. Since $f \circ g$ is a Borel function from $Y$ to $X$, by assumption the set $g^{-1}(f^{-1}(A))$ has the property of Baire in $Y$.
\end{proof}

\begin{corollary}\label{borel ősképre zárt topub és polishub}
    \p \; and topological-universally Baire sets are closed under Borel preimages, that is, if $A \s X$ is a \p/\topsp \; set and $X'$ is a Polish space and $f$ is a Borel function $f:X' \longrightarrow X$, then $f^{-1}(A)$ is \p/\topsp.
\end{corollary}

\begin{proof}
     By Theorem \ref{Borel preimages theorem} Polish-universally Baire $\Longleftrightarrow$ Polish-Borel-univer\-sal\-ly Baire and topological-universally Baire $\Longleftrightarrow \text{topological-Borel-universally Baire}$. Using the previous theorem which shows that $\mathcal{F}$-Borel-universally Baire sets are closed under Borel preimages, the proof is complete.
\end{proof}

\begin{remark}
The previous corollary also shows that being Polish-universally Baire only depends on the Borel structure, hence one can also define Polish-universally Baire sets in standard Borel spaces.
\end{remark}

We have seen in Corollary \ref{borel ősképre zárt topub és polishub} that $\topsp$ and $\p$ are closed under taking Borel preimages. This raises the question what the situation is with the $\compact$ and $\cm$ sets. Let us finish this section by proving that this is actually independent for $\cm$ sets, and the negative answer is consistent for $\compact$ sets.

\begin{theorem}
    Assuming CH there is a set $A \s \bs$ such that $A$ is compact-universally Baire and $\cm$, but $f^{-1}(A)$ does not have the property of Baire for some Borel function $f:2^\omega \longrightarrow \bs$.
\end{theorem}

\begin{proof}
    We choose $A$ to be a Luzin set as we did in Example \ref{ex:L}, where we showed that a Luzin set is $\compact$ and compact-metrizable-universally Baire, but $A$ itself does not have the property of Baire. We will show that there is a Borel function $f:2^\omega \longrightarrow \bs$ such that $f^{-1}(A)$ does not have the property of Baire in $2^\omega$, which will complete the proof of the theorem.

    Using Lemma \ref{perfectcontbairelemma} we identify $\bs$ with a dense $G_\delta$ subspace $N$ of $2^\omega$, then the set $A$ does not have the property of Baire in $2^\omega$. Finally, after fixing an $a \in N \setminus A$ we define the Borel function $f:2^\omega \longrightarrow N$ as:

    \begin{equation*}
f(x):=
\begin{cases*}
x, &  if  $x \in N,$ \\
a, &  if $x \in 2^\omega \setminus N.$
\end{cases*}
    \end{equation*}

Clearly, $f^{-1}(A)=A$ and we are done.
\end{proof}

In the compact metrizable case the full answer is actually an easy corollary.

\begin{theorem}
    It is consistent with ZFC that the compact metrizable-universally Baire sets are closed under taking Borel preimages.
\end{theorem}

\begin{proof}
    In Theorem \ref{t:equiv} we will prove that it is consistent with ZFC that $\cm$ and $\p$ sets are the same. By Corollary \ref{borel ősképre zárt topub és polishub}, $\p$ sets are closed under taking Borel preimages.
\end{proof}

However, the consistency of the positive direction for compact-universally Baire sets remains open.

\begin{question}
    Is it true in ZFC that the $\compact$ sets are not closed under taking Borel preimages?
\end{question}

\customsection{Definitions using Borel isomorphisms, changes of topologies and Boolean algebras}{Definitions using Borel isomorphisms, changes of topologies and Boolean algebras}{Borel isomorphisms and Boolean algebras}\label{Boolean section}

It has to be emphasized that in this section we are building heavily on ideas of Zakrzewski's papers \cite{Zak1}, \cite{Zak2}.

In \cite{Re}, Rec\l{}aw defined a subset $A \s \R^n$ to be universally Baire if for every Borel isomorphism $f:\R^n \to \R$, the image $f(A)$ has the property of Baire. A related notion is in \cite{Gr1}, where a set $A \subseteq X$ is called absolutely of the first category, if for every Borel isomorphism $f:X \longrightarrow X$ the set $f(A)$ is meager.


One of the main results of this section will be that we prove that Rec\l{}aw's definition coincides with the Polish-universally Baire sets, strengthening the statement that Polish-universal Baireness is the natural dual to universal measurability.

In this section we characterize the Polish-universally Baire sets in various ways. As in the above works of Zakrzewski, a surprising feature of the proofs will be that in order to prove that the notions introduced so far are equivalent, we need to introduce further equivalent definitions using Boolean algebras.

First we will need the notion of the category algebra:

\begin{definition}\label{def: cat algebra}
    We call a Boolean algebra $\mathbb{B}$ a category algebra, if it is complete, atomless and separable (has a countable dense set).
\end{definition} 

The following claim is probably well-known, however we could not find the explicit statement in the literature with a proof, so we have decided to include a short proof:

\begin{claim}\label{claim: unicity of the category algebra}
    The category algebra is unique up to isomorphism, that is, if we are given two category algebras $\mathbb{A}$ and $\mathbb{B}$, then they are isomorphic Boolean algebras.
\end{claim}

\begin{proof}
    Since both $\mathbb{A}$ and $\mathbb{B}$ are atomless and separable, it is easy to see that they contain dense countable atomless subalgebras $\tilde{\mathbb{A}}$, $\tilde{\mathbb{B}}$. 

    Then $\tilde{\mathbb{A}}$ and $\tilde{\mathbb{B}}$ are isomorphic Boolean algebras (see \cite[Exercise 7.22]{Jech}). We have that $\A$ and $\B$ are the completion of the Boolean algebras $\tilde{\mathbb{A}}$ and $\tilde{\mathbb{B}}$, and using that the completion of any Boolean algebra is unique up to isomorphism, the proof is complete. 
\end{proof}

By the previous claim, it makes sense to speak about the category algebra, that we will denote by $\cat$. The following lemma shows that the category algebra naturally appears in many situations.

\begin{lemma}\label{category algebra as borel mod meager}
    Suppose that $Y$ is a second countable topological space such that every point is meager and $Y$ is not meager in itself. Then
    $$
    \mathcal{B}\mathcal{P} (Y) / \mathcal{M}(Y) \simeq \mathcal{B}(Y) / \mathcal{B}(Y) \cap \mathcal{M}(Y) \simeq \cat.
    $$
\end{lemma}

\bp
First we prove that $\mathcal{B}\mathcal{P} (Y) / \mathcal{M}(Y) \simeq \mathcal{B}(Y) / \mathcal{B}(Y) \cap \mathcal{M}(Y)$. Suppose that we are given an equivalence class $[Z] \in \mathcal{B}\mathcal{P} (Y) / \mathcal{M}(Y)$, then  $Z$ can be written as $U \triangle M$, where $U$ is open and $M$ is meager. We define a Boolean algebra isomorphism as follows:
$$
\Phi:\mathcal{B}\mathcal{P} (Y) / \mathcal{M}(Y) \longrightarrow \mathcal{B}(Y)/ \mathcal{B}(Y) \cap \mathcal{M}(Y), \quad [Z] \mapsto [U].
$$

First, we need to check that $\Phi$ is well-defined, because it depends on which representative we choose from $[Z]$ and it also depends on how we write $Z$ as a symmetric difference of an open and a meager set. It is enough to check that if $Z^{'} \in  [Z]$ and $Z^{'} = U' \triangle M'$, where $U'$ is open and $M'$ is meager, then $U \triangle U'$ is meager. The set $U \triangle U'$ is meager because it is covered by
$(Z \triangle Z^{'}) \cup M \cup M'$.

In a similar fashion, one can easily show that $\Phi$ is bijective and preserves the order of the Boolean algebras, that is:
$$
[Z^{'}] \leq [Z] \Leftrightarrow \Phi([Z^{'}]) \leq \Phi([Z]).
$$

Since any bijective order preserving map between Boolean algebras is an isomorphism, we have proved that $\mathcal{B}\mathcal{P} (Y) / \mathcal{M}(Y) \simeq \mathcal{B}(Y) / \mathcal{B}(Y) \cap \mathcal{M}(Y)$.

Now we prove that $\mathcal{B}(Y) / \mathcal{B}(Y) \cap \mathcal{M}(Y) \simeq \cat$. We will prove that the Boolean algebra $\mathcal{B}(Y) / \mathcal{B}(Y) \cap \mathcal{M}(Y)$ is complete, atomless with a countable dense set, which will yield the second part of the theorem.

First, we show that $\mathcal{B}(Y) / \mathcal{B}(Y) \cap \mathcal{M}(Y)$ has a countable dense set. Let $\{V_i\}_{i \in \omega}$ be an enumeration of a basis of $Y$. We claim that the positive elements of $\{[V_i]\}_{i \in \omega}$ form a countable dense set in $\mathcal{B}(Y) / \mathcal{B}(Y) \cap \mathcal{M}(Y)$ (some of the $[V_i]$-s may be the $0$ element in the Boolean algebra, because we have not assumed that $Y$ is a Baire space).
Suppose that we are given a non-zero element $[B]$ in the Boolean algebra. Then choose an open set $U$ with $[U]=[B]$. Since $U$ is not meager, it contains a $V_i$ that is also not meager, hence $[V_i] \leq [U]=[B]$.

We prove now that $\mathcal{B}(Y) / \mathcal{B}(Y) \cap \mathcal{M}(Y)$ is complete. We will show that $\mathcal{B}(Y) / \mathcal{B}(Y) \cap \mathcal{M}(Y)$ is $\sigma$-complete and ccc (every partition is at most countable/every antichain is countable). It is well-known and easy to verify that taking the quotient of a $\sigma$-complete Boolean algebra with a $\sigma$-ideal is also a $\sigma$-complete Boolean algebra, hence $\mathcal{B}(Y) / \mathcal{B}(Y) \cap \mathcal{M}(Y)$ is $\sigma$-complete. The ccc property follows from the existence of a countable dense set. Every $\sigma$-complete and ccc Boolean algebra is complete, see \cite[Exercise 7.33.]{Jech}.

Finally, we prove that $\mathcal{B}(Y) / \mathcal{B}(Y) \cap \mathcal{M}(Y)$ is atomless. Suppose for a contradiction that there is an atom $[B] \in \mathcal{B}(Y) / \mathcal{B}(Y) \cap \mathcal{M}(Y)$. By fixing a representative $Z \in [B]$, since $[B]$ is an atom, we have that
$$
\forall D \in \mathcal{B}(Y): \textrm{ either } Z \cap D \in \mathcal{M}(Y) \textrm{ or } Z \setminus D \in \mathcal{M}(Y).
$$
Therefore, since all points of $Y$ are meager and $\mathcal{B}(Y) / \mathcal{B}(Y) \cap \mathcal{M}(Y)$ is $\sigma$-complete, every $x \in Z$ has a neighbourhood $V_x$ such that $V_x \cap Z$ is meager. We can choose countably many $V_x$ such that they cover $Z$, which implies that $Z$ is meager, a contradiction.
\ep

\begin{remark}
Now we give an example demonstrating that in the previous lemma it is not enough to suppose that the space $Y$ is perfect instead of that every point is meager.

    It is clear that if every point is meager in a space $Y$ then $Y$ is perfect, otherwise we would have a meager open singleton $\{y \}$, which is impossible.

    It is also easy to see that if every point is closed in a perfect space $Y$ then every point is nowhere dense, hence every point is meager.

     Consider now the following topology $\tau$ on $\R$:

    $$
    \tau:= \left\{ \left. \left(-\frac{1}{n}, \frac{1}{n} \right) \right| \; n \in \omega \setminus \{0\} \right\} \cup \{\R\} \cup \{ \emptyset \}.
    $$

    Clearly, $(\R,\tau)$ is second countable and no points are isolated. On the other hand, one can easily check that $\{0\}$ is a nonmeager point with the property of Baire, hence $\{0\}$ is an atom in the Boolean algebra $\mathcal{BP}(\R,\tau)/ \mathcal{M}(\R, \tau)$.
\end{remark}

Before the following lemma, recall that if $\B$ is a complete Boolean algebra and $\A$ is a subalgebra of $\B$, then we say that $\A$ is a \emph{complete subalgebra} of $\B$ if $\A$ is a complete Boolean algebra and the $\sup$ of every subset $X$ of $\A$ is the same in $\A$ and $\B$.

\begin{lemma}\label{lemma: complete subalgebra of the category algebra}
A complete atomless subalgebra of $\cat$ is isomorphic to $\cat$.    
\end{lemma}

\begin{proof}
    Take an arbitrary complete atomless subalgebra $\B$ of $\cat$. By Claim \ref{claim: unicity of the category algebra} it is enough to prove that $\B$ is separable.

    Let $D:=\{d_1,d_2,...\}$ be a countable dense set for $\cat$. For every $d_i$, we define $b_i \in \B$ in the following way:

    $$
    b_i:= \inf\{c \in \B \mid c \geq d_i \}.
    $$

Notice that $d_i \leq b_i$ for every $i$. We claim that $\{b_i \mid i \in \omega \}$ is a countable dense set for $\B$. Fix an arbitrary positive element $b \in \B$. Since the set $D$ is dense, there exists a $d_i \in D$ such that $d_i \leq b$. Then by definition $b_i \leq b$, which completes the proof.    
\end{proof}

In the following theorem, as explained above, both the notations $\mathcal{B}(X)$ and  $\mathcal{B}(X, \tau)$ will denote the class Borel subsets of $(X, \tau)$. We will follow our convention that if a statement or proof only involves one topology then we simply write $\mathcal{B}(X)$.

\bt\label{thm: boole equivalence}
For a set $A$ in a Polish space $(X,\tau)$, the following are equivalent:

\begin{enumerate}[label=(\roman*)]
    \item\label{polishunivbaire} $A$ is Polish-universally Baire.

    \item\label{Borel isomorphic preimage} For every Polish space Y and Borel isomorphism $f:Y \longrightarrow X$ onto $X$, the set $f^{-1}(A)$ has the property of Baire, equivalently, for every Polish topology $\tau'$ on $X$ such that $\mathcal{B}(X, \tau)=\mathcal{B}(X,\tau')$ we have that $A\in\mathcal{B}\mathcal{P}(X,\tau')$.

    \item\label{Borel isomorphic preimage local} For every Polish space Y, subset $Z \s Y$ and Borel isomorphism $f:Z \longrightarrow A$ onto $A$, the set $Z=f^{-1}(A)$ has the property of Baire.

    \item\label{B(X) modulo I = C} For every $\sigma$-ideal $\mathcal{I} \s \mathcal{B}(X,\tau)$ such that $\mathcal{B}(X,\tau)/\mathcal{I} \simeq \cat$ there are Borel sets $ C,D \in \mathcal{B}(X,\tau)$ such that $C \s A \s D$ and $D \setminus C \in \mathcal{I}$.

    \item\label{B(A) modulo I = C} For every $\sigma$-ideal $\mathcal{I} \s \mathcal{B}(A,\tau \vert_A)$ such that $\mathcal{B}(A,\tau\vert_A) / \mathcal{I} \simeq \cat $ there exists a Borel set $B \in \mathcal{B}(X,\tau)$ such that $B \s A$ and $A \setminus B \in \mathcal{I}$.

    \item\label{M2-UB} The set $A$ is second countable-universally Baire, that is, $f^{-1}(A)$ has the property of Baire in every second countable space $Y$ for every continuous $f:Y \longrightarrow X$.

    \item\label{M2 topology on X} The set $A$ has the property of Baire in every second countable topology $\tau'$ on $X$ such that $\mathcal{B}(X,\tau) \s \mathcal{B}\mathcal{P}(X,\tau')$.

\item\label{M2 topology on A} For every second countable topology $\tau'$ on $A$ such that $\mathcal{B}(A,\tau \vert_A) \s \mathcal{BP}(A,\tau')$  we can write $A=C \sqcup D$ where $C \in \mathcal{B}(X,\tau)$ and $D \in \mathcal{M}(A,\tau')$.

    \item\label{Polish topology refinement} The set $A$ has the property of Baire in every finer Polish topology, that is, for every Polish $\tau' $ on $X$ with $\tau \s \tau'$ we have that $A \in \mathcal{B}\mathcal{P}(X,\tau')$.

\end{enumerate}

\et

\begin{proof}

We will prove the following implications: $\ref{Borel isomorphic preimage} \Rightarrow \ref{B(X) modulo I = C} \Rightarrow \ref{M2-UB} \Rightarrow \ref{M2 topology on X} \Rightarrow \ref{Borel isomorphic preimage}$,  $\ref{B(X) modulo I = C} \Rightarrow \ref{B(A) modulo I = C} \Rightarrow \ref{M2 topology on A}  \Rightarrow \ref{Borel isomorphic preimage local} \Rightarrow \ref{Borel isomorphic preimage}$, $\ref{M2-UB} \Rightarrow \ref{polishunivbaire} \Rightarrow \ref{Borel isomorphic preimage}$, $\ref{Polish topology refinement} \Leftrightarrow \ref{Borel isomorphic preimage}$.

$\ref{Borel isomorphic preimage} \Rightarrow \ref{B(X) modulo I = C}:$  Suppose that $\mathcal{I}$ is a $\sigma$-ideal on $\mathcal{B}(X)$ such that $\mathcal{B}(X) / \mathcal{I} \simeq \cat$.  We will show that there exists a Polish topology $\tau'$ on $X$ such that $\mathcal{B}(X,\tau') = \mathcal{B}(X,\tau)$ and $\mathcal{M}(X,\tau') \cap \mathcal{B}(X,\tau') = \mathcal{I}$. We are going to use \cite[Theorem 15.10]{Ke}, which says that if we are given Polish spaces $Z_1,Z_2$ and $\sigma$-ideals $\mathcal{J}_1 \s \mathcal{B}(Z_1)$, $\mathcal{J}_2 \s \mathcal{B}(Z_2)$ with $\mathcal{B}(Z_1)/ \mathcal{J}_1 \simeq \mathcal{B}(Z_2)/ \mathcal{J}_2$ such that $\mathcal{J}_1$ and $\mathcal{J}_2$ contain uncountable sets, then there exists a Borel isomorphism $\varphi:Z_1 \longrightarrow Z_2$ onto $Z_2$ such that $A \in \mathcal{J}_1 \Longleftrightarrow \varphi(A) \in \mathcal{J}_2$. We will apply this with $\mathcal{J}_1 = \mathcal{I}$ and $\mathcal{J}_2 = \mathcal{B}(\bs) \cap \mathcal{M}(\bs)$.

We first show that  $\mathcal{B}(X) /\mathcal{I} \simeq \mathcal{B}(\bs) / (\mathcal{B}(\bs) \cap \mathcal{M}(\bs) )$, in fact, they are both isomorphic to $\cat$. Indeed, this is clear by assumption for $\mathcal{B}(X) / \mathcal{I}$, and this is also clear for $\mathcal{B}(\bs) / (\mathcal{B}(\bs) \cap \mathcal{M}(\bs) )$ by Lemma \ref{category algebra as borel mod meager}.

Now we check that the ideals contain uncountable sets. This is clear for $\mathcal{B}(\bs) \cap \mathcal{M}(\bs)$. For $\mathcal{I}$ we argue as follows. Suppose that $\mathcal{I}$ contains only countable sets. Since it is easy to construct continuum many pairwise disjoint Borel sets of size continuum, there exist continuum many nonzero pairwise disjoint elements in the factor algebra $\mathcal{B}(X)/ \mathcal{I}$, contradicting the ccc property of $\cat$. Putting everything together we get a Borel isomorphism $\varphi:X \longrightarrow \bs$ such that $$A \in \mathcal{I} \Longleftrightarrow \varphi(A) \in \mathcal{B}(\bs) \cap \mathcal{M}(\bs).$$
The appropriate topology $\tau'$ on $X$ is defined as the topology of the $\bs$ 'pulled back' by $\varphi$, that is, for $O \s X$:

$$
O \in \tau' \Longleftrightarrow \varphi(O) \text{ is open in } \bs.
$$
With this definition of $\tau'$, we have $ \mathcal{M}(X,\tau')\cap \mathcal{B}(X,\tau') = \mathcal{I}$.

It follows from $\ref{Borel isomorphic preimage}$ that the set $A$ has the property of Baire with respect to the topology $\tau'$, which implies that there exist Borel sets $C,D$ such that $C \s A \s D$ and $D \setminus C  \in \mathcal{M}(X,\tau')\cap \mathcal{B}(X,\tau') = \mathcal{I}$.

$\ref{B(X) modulo I = C} \Rightarrow \ref{M2-UB}:$ Take an arbitrary second countable space $Y$ and a continuous function $f:Y \longrightarrow X$, we need to prove that $f^{-1}(A)$ has the property of Baire. Let us consider the following collection of open sets:

$$
\mathcal{H}:=\{ V \s Y \mid V \textit{ is an open subset in } Y, \; f(V) \textit{ is countable} \}.
$$

Using that the space $Y$ is second countable, we see that $W:=\cup \mathcal{H}$ is an open set such that $f(W)$ is countable, which implies that $f^{-1}(A) \cap W$ is a relative $F_\sigma$ set in $W$.

Let $Z:= Y \setminus W$. To deduce that $f^{-1}(A)$ has the property of Baire, it is enough to prove that $f^{-1}(A) \cap Z$ has the property of Baire in $Y$ and by Lemma \ref{bairetultranz} it is enough to prove that $f^{-1}(A) \cap Z$ has the property of Baire in $Z$. 

If $Z$ is meager in itself (including the case when $Z=\emptyset$) then we are done. So we can assume that $Z$ is not meager in itself and let us denote by $g$ the restriction of $f$ to $Z$. We claim that $g$ is nowhere locally constant. Indeed, suppose for a contradiction that $g$ is constant on a nonempty open set $U \s Z$. Then there exists an open set $V$ in $Y$ such that $V \cap Z =U$ which implies that $f(V)$ is countable, so $U \s W$, a contradiction. Using that $g$ is nowhere locally constant, it follows that every point $x \in Z$ is meager in $Z$. All in all, by Lemma \ref{category algebra as borel mod meager} we have that $\B:= \mathcal{B}(Z)/ \mathcal{B}(Z) \cap \mathcal{M}(Z)$ is isomorphic to $\cat$.

We define the following $\sigma$-ideal $\mathcal{I}$ on $\mathcal{B}(X):$

$$
\mathcal{I}:= \{ B \in \mathcal{B}(X) \mid g^{-1}(B) \in \mathcal{M}(Z) \}.
$$
It is easy to see that $\A:= \mathcal{B}(X) / \mathcal{I}$ is a $\sigma$-complete Boolean algebra and the ccc property of $\A$ follows from the ccc property of $\mathbb{B}$, hence $\mathbb{A}$ is a complete Boolean algebra.

Let us recall that an embedding between Boolean algebras is called a $\sigma$-complete embedding between Boolean algebras if it preserves the sup of countably many Boolean algebra elements (assuming in advance that both Boolean algebras are $\sigma$-complete). We use the function $g$ to define a $\sigma$-complete embedding of the Boolean algebra $\A$ into the Boolean algebra $\B$ by the formula:

$$
\Gamma:[B] \in \mathcal{B}(X) / \mathcal{I} \mapsto [g^{-1}(B)]  \in  \mathcal{B}(Z)/(\mathcal{M}(Z) \cap \mathcal{B}(Z)).
$$

Indeed, $\Gamma$ is an embedding and it is clearly $\sigma$-complete. We prove that $\Gamma$ is actually a complete embedding of $\A$ into $\B$, hence $\A$ can be viewed as a complete subalgebra of $\B$. Recall that a Boolean subalgebra $\Tilde{\mathbb{A}}$ of a Boolean algebra $\Tilde{\mathbb{B}}$ is called a regular subalgebra if every maximal antichain in $\Tilde{\mathbb{A}}$ is a maximal antichain in $\Tilde{\mathbb{B}}$. By \cite[Exercise 7.31]{Jech}, if $\Tilde{\mathbb{A}}$ is a regular subalgebra of $\Tilde{\mathbb{B}}$, then the inclusion map extends to a unique complete embedding of the completion of $\Tilde{\mathbb{A}}$ into the completion of $\Tilde{\mathbb{B}}$. Using the previously mentioned result, it is enough to prove that for every maximal antichain $H \s \A$, the set $\Gamma(H)$ forms a maximal antichain in $\B$. But this is easy since, $\A$ is ccc, hence the set $H$ is countable and the sup of the countable set $\Gamma(H)$ in $\B$ is $1_\B$ because $\Gamma$ is a $\sigma$-complete embedding of $\A$ into $\B$.

We will now prove that $\A$ is atomless. Suppose for a contradiction that there exists some $[B] \in \A$ which is an atom. Then for every $C \in \mathcal{B}(X)$ one of the following hold: $B \cap C \in \mathcal{I}$ or $B \setminus C \in \mathcal{I}$. Notice that for every point $x \in X$ we have that $\{x\} \in \mathcal{I}$ since $g^{-1}(x)$ is meager in $Z$ (recall that $g$ is nowhere locally constant). We claim that for every $x \in B$ there exists some neighbourhood $V_x$ such that $B \cap V_x \in \mathcal{I}$. If this is not the case, then for some $x$ for every $\frac{1}{n}$-neighbourhood $U_\frac{1}{n}$ of $x$ we have that $U_\frac{1}{n}^c \cap B \in \mathcal{I}$, which implies that $B \setminus \{x\} \in \mathcal{I}$, a contradiction. For every $x \in B$ fix a neighbourhood $V_x$ such that $B \cap V_x \in \mathcal{I}$, choosing a countable cover of $B$ from $\{V_x \mid x \in B\}$ yields that $B \in \mathcal{I}$, a contradiction.

By Lemma \ref{lemma: complete subalgebra of the category algebra}, the Boolean algebra $\A$ is isomorphic to $\cat$ so by $\ref{B(X) modulo I = C}$ there exist $C,D \in \mathcal{B}(X)$ such that $C \s A \s D$ and $D \setminus C \in \mathcal{I}$. The latter implies that $g^{-1}(A)$ has the property of Baire because $g^{-1}(A) = g^{-1}(C) \cup g^{-1}(A \setminus C)$ and $g^{-1}(C)$ is Borel and $g^{-1}(A \setminus C)$ is meager (since it is covered by the set $g^{-1}(D \setminus C)$ which is meager by definition).

$\ref{M2-UB} \Rightarrow \ref{M2 topology on X}:$ Suppose that $\tau'$ is a second countable topology on $X$ such that $\mathcal{B}(X,\tau) \s \mathcal{BP}(X,\tau')$, we need to prove that $A$ has the property of Baire in the topology $\tau'$. Let us consider the identity function on $X$ in the following sense:
$$
id:(X,\tau') \longrightarrow (X,\tau).
$$

Then $id$ is a Baire measurable function by assumption so by \cite[Theorem 8.38.]{Ke} it is continuous restricted to a dense $G_\delta$ subspace $Y$ of $(X,\tau')$. Using that $A$ is second countable-universally Baire we get that $A \cap Y$ has the property of Baire in $Y$ and by Lemma $\ref{bairetultranz}$ the set $A \cap Y$ has the property  of Baire in $(X,\tau')$. The complement of $Y$ is meager, which finishes the proof.

$\ref{M2 topology on X} \Rightarrow \ref{Borel isomorphic preimage}:$ Suppose that we are given a Polish space $Y$ and a Borel isomorphism $f:Y \longrightarrow X$ onto $X$, we will show that $f^{-1}(A)$ has the property of Baire in $Y$. We define a topology $\tau'$ on $X$ such that for $U \s X$:

$$
U \in \tau' \overset{\textit{def}}{\Longleftrightarrow} f^{-1}(U) \textit{ is open in } Y.
$$

Since the function $f$ is a bijection, the Polish space $Y$ is homeomorphic to $(X,\tau')$, furthermore, $\mathcal{B}(X, \tau') = \mathcal{B}(X,\tau)$ which implies that $\mathcal{B}(X,\tau) \s \mathcal{BP}(X,\tau')$. It follows from $\ref{M2 topology on X}$ that $A$ has the property of Baire in $(X, \tau')$, so $f^{-1}(A)$ has the property of Baire in $Y$.

$\ref{B(X) modulo I = C} \Rightarrow \ref{B(A) modulo I = C}:$  Suppose that we are given a $\sigma$-ideal $\mathcal{I}$ on the Borel sets of $A$ such that $\mathcal{B}(A) / \mathcal{I} \simeq \cat$, then we define the following $\sigma$-ideal $\mathcal{J}$ on the Borel sets of $X$:
$$
\mathcal{J}:=\{B \in \mathcal{B}(X) \mid B \cap A \in \mathcal{I} \}.
$$

It is easy to see that $\mathcal{B}(X) / \mathcal{J} \simeq \mathcal{B}(A) / \mathcal{I} \simeq \cat$, hence there exist $C,D \in \mathcal{B}(X)$ such that $C \s A \s D$ and $D \setminus C \in \mathcal{J}$. Then $(D \setminus C) \cap A \in \mathcal{I}$, so $A \setminus C \in \mathcal{I}$, and we are done.

$\ref{B(A) modulo I = C} \Rightarrow \ref{M2 topology on A}:$ Suppose that we are given a second countable topology $\tau'$ on $A$ such that $\mathcal{B}(A, \tau \vert_A) \s \mathcal{BP}(A, \tau')$, we will prove that $A$ can be written as a $B \sqcup C$, where $B \in \mathcal{B}(X,\tau)$ and $C \in \mathcal{M}(A, \tau')$. As a first step, we get rid of the nonmeager points:
$$
\mathcal{H}:=\{U \in \tau' \mid \exists p \in U \textit{ such that } \{p\} \textit{ is $\tau'$-comeager in } U \textit{ and } U \textit{ is $\tau'$-nonmeager} \}.
$$

Notice that if $U \in \mathcal{H}$ then the witnessing $p$ is unique, so from now on we refer to a $U \in \mathcal{H}$ as $U_p$. Using that $\tau'$ is second countable, we can select countably many $U_p$-s such that:
$$
\cup\mathcal{H} = \underset{i \in I}{\cup} U_{p_i} \quad (\textit{where } I \textit{ is a countable set}).
$$

Let $Z:= A \setminus \underset{i \in I}{\cup} U_{p_i}$. First suppose that $Z \in \mathcal{M}(A,\tau')$.
We define the sets $C$ and $D$ as follows: 
$$
C:=\{p_i \mid i \in I \}, \quad D:= A \setminus C.
$$
Since the set $C$ is countable, we have that $C \in \mathcal{B}(X)$ and it is easy to see that $D \s \underset{i \in I}{\cup} (U_{p_i} \setminus \{p_i\}) \cup Z$, which is $\tau'$-meager. This finishes the proof in the case when $Z$ is $\tau'$-meager.

So let us suppose that the $\tau'$-closed set $Z$ is $\tau'$-nonmeager, then the set 
$$
W:= \inter_{\tau'} Z
$$
is $\tau'$-nonmeager, since the boundary of a closed set is nowhere dense. Hence, $W$ is also $\tau'\vert_W$-nonmeager. We claim that every point is $\tau' \vert_W$-meager in $W$. Indeed, suppose for a contradiction that there is a point $p \in W$ which is not $\tau' \vert_W$-meager, then $\{p\}$ has the property Baire and there exists a $\tau'$-nonmeager open set $U \s W$ such that $\{p\} \vartriangle U$ is $\tau' \vert_W$-meager. The latter implies that $p \in U$ and $p$ is $\tau' \vert_W$-comeager in the $\tau' \vert_W$-nonmeager open set $U$ which implies that $U \in \mathcal{H}$, hence $p \in \underset{i \in I}{\cup} U_{p_i}$, contradicting that $p \in W$. The space $W$ is second countable, hence by Lemma \ref{category algebra as borel mod meager} we obtain that the Boolean algebra
$$
\B:= \mathcal{BP}(W, \tau' \vert_W)/ \mathcal{M}(W, \tau' \vert_W) \simeq \cat.
$$

Now let us define a $\sigma$-ideal $\mathcal{I}$ on $\mathcal{B}(A,\tau \vert_A)$ as follows:
$$
\mathcal{I}:=\{ B \in \mathcal{B}(A,\tau \vert_A) \mid B \cap W \in \mathcal{M}( W, \tau' \vert_W) \}.
$$

Our goal is to prove that $\A:= \mathcal{B}(A) / \mathcal{I} \simeq \cat$. As a first step, we embed $\A$ into $\B$ as follows:
$$
\Gamma:[B] \in \mathcal{B}(A) / \mathcal{I} \mapsto [B \cap W]  \in \mathcal{BP}(W, \tau' \vert_W)/ \mathcal{M}(W, \tau' \vert_W).
$$

It is easy to see that $\Gamma$ is well-defined and $\Gamma$ is actually a $\sigma$-complete embedding of $\A$ into $\B$. From now on we identify $\A$ with its $\Gamma$ image in $\B$. Since the Boolean algebra $\B$ is ccc, we obtain that $\A$ is also ccc, and using that $\A$ is a $\sigma$-complete Boolean algebra we get that $\A$ is a complete Boolean algebra. 

We prove now that $\A$ is a complete subalgebra of $\B$ (we use the same argument as in $\ref{B(X) modulo I = C} \Rightarrow \ref{M2-UB}$). By Exercise 7.31 in \cite{Jech} if $\Tilde{\mathbb{A}}$ is a regular subalgebra of $\Tilde{\mathbb{B}}$, then the inclusion map extends to a unique complete embedding of the completion of $\Tilde{\mathbb{A}}$ to the completion of $\Tilde{\mathbb{B}}$. Using the previously mentioned result, it is enough to prove that for every maximal antichain $H \s \A$, the set $H$  forms a maximal antichain in $\B$, too. But this is easy since $\A$ is ccc, hence $H$ is a countable set and the sup of the countable set $H$ in $\B$ and in $\A$ is the same (namely the maximal element $1_\B$ of the Boolean algebra) because $\Gamma$ is a $\sigma$-complete embedding of $\A$ into $\B$.

We prove now that $\A$ is atomless. Suppose for a contradiction that 
$\A$ has an atom, that is, there exists $B \in \mathcal{B}(A)$ such that $B \notin \mathcal{I}$ but for every $B' \in \mathcal{B}(A)$ either $B \cap B' \in \mathcal{I}$ or $B \setminus B' \in \mathcal{I}$. We claim that for every $x \in B \cap W$ there exists a $\tau \vert_A$-open neighbourhood $V_x$ of $x$ such that $B \cap V_x \in \mathcal{I}$. Indeed, suppose for a contradiction that there exists a point $x \in B \cap W$ such that for every $\tau \vert_A$-open neighbourhood $V$ of $x$ we have that $B \cap V \notin \mathcal{I}$. If we fix a countable $\tau \vert_A$-open neighbourhood basis $\{V_n\}_{n \in \omega}$ of $x$ we have that:
$$\underset{n \in \omega}{\cup} (B \setminus V_n) = B \setminus \{x\} \in \mathcal{I}.$$

This implies that $\{x\} \notin \mathcal{I}$ which contradicts the fact that every point in $W$ is meager.

So according to the previous claim, let us fix for every $x \in B \cap W$ an open neighbourhood $V_x$ of $x$ such that $B \cap V_x \in \mathcal{I}$. Using that the topology $\tau \vert_A$ is second countable, we can select a countable cover $\{V_{x_i}\}_{i \in \omega}$ of $B \cap W$ which immediately implies that $B \in \mathcal{I}$, a contradiction.

All in all, we have proved that $\A=\mathcal{B}(A)/ \mathcal{I}$ is isomorphic to a complete atomless subalgebra of $\cat$, hence by Lemma \ref{lemma: complete subalgebra of the category algebra} the Boolean algebra $\mathcal{B}(A)/ \mathcal{I} \simeq \cat$. Applying the assumption we obtain a (global) Borel set $C' \in \mathcal{B}(X,\tau)$ and a (local) Borel set $D' \in \mathcal{B}(A,\tau\vert_A)$ such that $A = C' \sqcup D'$ and $D' \in \mathcal{I}$. Let us define $C$ and $D$ as follows:
$$
C:=C' \cup \{p_i\mid i \in I \}, \quad \quad \quad D:= A \setminus C.
$$

It is clear that $C \in \mathcal{B}(X)$, $D \in \mathcal{B}(A)$ and $A=C \sqcup D$. Let us show that $D \in \mathcal{M}(A,\tau')$. First, note that $D \cap W = D \cap W'$ and $A$ is the disjoint union of $\underset{i \in I}{\cup} U_{p_i}$, $\partial_{\tau'}W$ and $W$, hence the set $D$ can be covered as follows:
$$
D \s \underset{i \in I}{\cup} (U_{p_i} \setminus \{p_i\}) \cup \partial_{\tau'} W  \cup (D' \cap W ).
$$

The sets $\underset{i \in I}{\cup} (U_{p_i} \setminus \{p_i\})$ and $\partial_{\tau'} W$ are meager (in the topology $\tau')$, because the first set is a countable union of meager sets and the second set is the boundary of an open set. To see that $D' \cap W$ is also meager, it is enough to note that $D' \in \mathcal{I}$, which is equivalent to $D' \cap W \in \mathcal{M}(W, \tau' \vert_W)$ and this implies that $D' \cap W \in \mathcal{M}(A, \tau')$. We have proved that $D \in \mathcal{M}(A,\tau')$, which finishes the proof.

$\ref{M2 topology on A} \Rightarrow \ref{Borel isomorphic preimage local}:$  Take an arbitrary Polish space $Y$ and a subset $Z \s Y$ with a Borel isomorphism $f:Z \longrightarrow A$ onto $A$, we have to prove that $Z$ has the property of Baire in $Y$. We define a topology $\tau'$ on $A$ such that for $U \s A$:
$$
U \in \tau' \overset{\textit{def}}{\Longleftrightarrow} f^{-1}(U) \textit{ is open in } Z.
$$

The space $(A, \tau')$ is homeomorphic to $Z$ (witnessed by the function $f$), hence $\tau'$ is second countable. Also note that $\mathcal{B}(A)= \mathcal{B}(A, \tau') \s \mathcal{BP}(A, \tau')$, so by assumption the set $A$ can be written as $A= C \sqcup D$ where $C \in \mathcal{B}(X)$ and $D \in \mathcal{M}(A, \tau')$. Since the function $f$ is a Borel isomorphism, the set $f^{-1}(C) \in \mathcal{B}(Y)$, on the other hand, $f$ is a homeomorphism (between the appropriate spaces), hence $f^{-1}(D) \in \mathcal{M}(Z) \s \mathcal{M}(Y)$. All in all, the set $Z$ can be written  as the union of a Borel set $f^{-1}(C)$ and a meager set $f^{-1}(D)$ (in $Y$), hence the set $Z$ has the property of Baire in $Y$.

$\ref{Borel isomorphic preimage local} \Rightarrow \ref{Borel isomorphic preimage}:$ Take an arbitrary Polish space $Y$ and a Borel isomorphism $f:Y \longrightarrow X$ onto $X$, we shall prove that $f^{-1}(A)$ has the property of Baire in $Y$. If $g$ denotes the restriction of $f$ to $f^{-1}(A)$ then $g$ is a Borel isomorphism between $f^{-1}(A)$ and $A$, hence $f^{-1}(A)$ has the property of Baire in $Y$.

$\ref{M2-UB} \Rightarrow \ref{polishunivbaire}:$ This implication is clear because every Polish space is second-countable.

$\ref{polishunivbaire} \Rightarrow \ref{Borel isomorphic preimage}:$ By Theorem \ref{Borel preimages theorem} we know that every Polish-universally Baire set is also Polish-Borel-universally Baire, hence if $A \s X$ is Polish-universally Baire and we have a Borel isomorphism $f:Y \longrightarrow X$ then $f^{-1}(A)$ has the property of Baire in $Y$.

$\ref{Borel isomorphic preimage} \Rightarrow \ref{Polish topology refinement}:$ If we take an arbitrary Polish topology $\tau'$ on $X$ such that $\mathcal{B}(X)= \mathcal{B}(X,\tau')$ (including the case, when $\tau' \supseteq \tau$), then the identity map from $(X,\tau'$) to $(X,\tau)$ is a Borel isomorphism, so by assumption the set $A$ has the property of Baire in $(X,\tau')$.

$\ref{Polish topology refinement} \Rightarrow \ref{Borel isomorphic preimage}:$ Consider an arbitrary Polish space $Y$ and a Borel isomorphism $f:Y \longrightarrow X$ onto $X$, we have to show that $f^{-1}(A)$ has the property of Baire in $Y$. As a first step, define the topology $\overline{\tau}$ on $X$ such that for $U \s X:$
$$
U \in \overline{\tau} \overset{\textit{def}}{\Longleftrightarrow} f^{-1}(U) \textit{ is open in } Y.
$$

It is clear that $Y$ is homeomorphic to the space $(X,\overline{\tau})$ (witnessed by the function $f$) and it is enough to prove that the set $A$ has the property of Baire in $(X,\overline{\tau})$. Consider the identity function as:
$$
id: (X,\overline{\tau}) \longrightarrow (X,\tau).
$$
Since $id$ is Borel, there exists a dense $G_\delta$ subspace $G$ in the topology $\overline{\tau}$ such that $id \vert_G$ is continuous (see \cite[Theorem 8.38.]{Ke}), in other words, $\overline{\tau} \vert_G \supseteq \tau \vert_G$. Since the set $G$ is Borel in the topology $\tau$, there exists a Polish topology $\sigma$ on $X$ such that $\sigma \supseteq \tau$ and $G$ is a clopen set in $\sigma$ (see \cite[Theorem 13.1]{Ke}). 

Let us define the topology $\tau'$ on $X$ as the disjoint union of $(G,\overline{\tau} \vert_G)$ and $(G^c, \sigma \vert_{G^c})$. The topology $\tau'$ is Polish as it was obtained as a disjoint union of two Polish topologies and we also have that $\tau' \supseteq \tau$. By assumption, the set $A$ has the property of Baire in $(X,\tau')$. Using that $(X,\tau)$ is a disjoint union, the set $A \cap G$ has the property of Baire in $(G, \overline{\tau} \vert_G)$. By Lemma \ref{bairetultranz} the set $A \cap G$ has the property of Baire in $(X, \overline{\tau})$ and using that the set $G$ is comeager in $(X,\overline{\tau})$ we can conclude that $A$ has the property of Baire in $(X,\overline{\tau})$, which finishes the proof.
\end{proof}


Finally, let us finish this section with two examples, where we show that (local) Borel injective images/preimages cannot describe Polish-universally Baire sets, hence, the Borel isomorphism cannot be changed to Borel injection in Theorem \ref{thm: boole equivalence} \ref{Borel isomorphic preimage local}.

The first example shows that there is a Polish-universally Baire set such that it has a Borel injective preimage that does not have the property of Baire.

\begin{ex}
    There exists a set $B \s [0,1)$ and a Borel injection $f:B \longrightarrow [0,1)$ such that $f^{-1}([0,1))=B$ and $B$ does not have the property of Baire in $[0,1)$.
\end{ex}

\begin{proof}
    Let $f(x) = 2x \; mod \; 1$ and let $B' \s [0,\frac{1}{2})$ be any set without the property of Baire and set $B:=B' \cup (([0,\frac{1}{2}) \setminus B') + \frac{1}{2})$.
\end{proof}

The second example shows that there exists a Polish-universally Baire set such that it has a Borel injective image that is not Polish-universally Baire.

\begin{ex}
    There exists a Polish-universally Baire set $A \s [0,1]^2$ and a continuous injection $f:A \longrightarrow [0,1]$ such that $f(A)$ does not have the property of Baire.
\end{ex}

\begin{proof}
    By \cite[Theorem 1.]{Grzunivmeager} there exists a universally meager set $B \s [0,1]$ of cardinality $\non(\mathcal{M})$ and let us fix a set $C \s [0,1]$ of cardinality $\non(\mathcal{M})$ such that $C$ does not have the property of Baire. Let $f$ be an arbitrary bijection from $B$ to $C$ and let $A:=\text{graph($f$)}$ (hence $A \s [0,1]^2$). 

    We will prove that $A$ is universally meager. Let us recall \cite[Theorem 2.1.]{Zak1} which says that a subset $M \subseteq X$ of a perfect Polish space $X$ is universally meager if and only if there is no nonmeager set $D \subseteq X$ and Borel injection $g: D \to M$. Building on the standard Peano curve constructions, it is easy to fix a Borel bijection $h:[0,1] \to [0,1]^2$ such that $Z \s [0,1]$ is meager if and only if $h(Z) \subseteq [0,1]^2$ is meager. Suppose for a contradiction that $A$ is not universally meager, then by the above theorem there exists a nonmeager $D \s [0,1]^2$ and a Borel injection $g:D \to A$. If $\pi_1$ denotes the projection from $[0,1]^2$ on the first coordinate, then $\pi_1 \circ g \circ h : h^{-1}(D) \to  B$ is a Borel injection from the nonmeager $h^{-1}(D)$ into $B$, a contradiction.
    
    Finally, notice that $C$ is the continuous injective image of $A$, since $C = \pi_2(A)$, where $\pi_2$ denotes the projection from $[0,1]^2$ on the second coordinate.
\end{proof}

\customsection{Simultaneous reflection of everywhere nonmeagerness}{Simultaneous reflection of everywhere nonmeagerness}{Reflection of everywhere nonmeagerness}\label{Reflection section}

In this section we take a detour and prove a seemingly unrelated result that we will need in the proof of the fact that an implication above can consistently be reversed. 

Bartoszy\'nski \cite{Ba} proved that the following statement is consistent: if a set $A \subseteq \N^\omega$ is nonmeager then there exists a nowhere dense Cantor set $C \s \bs$ such that $A \cap C$ is also nonmeager relative to $C$. (In fact, Bartoszy\'nski only states that $C$ is compact, which is a typo, but he actually produces a nowhere dense Cantor set.) For two more proofs of this theorem see Burke-Miller \cite{Miller} and Szládek \cite{Szl}, and for some strengthenings see Ciesielski-Shelah \cite{CieShe} and Elekes-Stepr\=ans \cite{ElekStep}. What we will need is a slightly modified version of Bartoszy\'nski's result. On the one hand, we need to replace nonmeagerness by everywhere nonmeagerness, and on the other hand, we need to find a Cantor set that works simultaneously for more than one everywhere nonmeager sets. The proof we modify is the one by Szládek, and since it is unpublished and only exists in Hungarian, we have decided to write out all details for completeness.

\subsection{A single step in the iteration}

We will use the \emph{Miller forcing}, which we denote by $\mathbb{M}$. The elements of $\mathbb{M}$ are nonempty subtrees of $\mathbb{N}^{<\omega}$ with some extra properties, ordered under inclusion. For $p \in \mathbb{M}$ we require that for all $\sigma \in p$ the set $\rakovetkezo_p(\sigma)$ of immediate successors has cardinality either 1 or $\omega$, and $\forall \sigma \in p \; \exists \tau \in p$ such that $\tau \succ \sigma$ and $\tau$ has infinitely many immediate successors. Furthermore, for some technical reasons we will also require for a Miller tree $p$ to be \emph{monotone}, that is, $\sigma(k) < \sigma(k+1)$ holds $\forall \sigma \in p$ and $\forall k,k+1 \in \dom(\sigma)$. It is easy to see that adding this monotonicity assumption yields an isomorphic poset. If a node $\sigma \in p$ has infinitely many immediate successors, then we call that node a \emph{branching node} of $p$. The minimal (in the tree ordering) branching node of a $p \in \mathbb{M}$ is denoted by $\stem(p)$.

If we have a Miller tree $p$, then by \emph{thinning out $p$ at height $h$ for the value $j$} we mean the following tree:
$$
\{\sigma \in p \mid \exists \tau \in p, \tau \succ \sigma,  \abs{\tau} > h, \tau(h) > j \}.
$$

It is easy to see that if we thin out a Miller tree then the result is also a tree. Furthermore, it is easy to verify (using monotonicity) that if we thin out $p$ at height $h \geq \abs{\stem(p)}$, then the resulting tree is again a Miller tree and every branching node $\sigma \in p$ with $\abs{\sigma} \leq h$ remains a branching node of the resulting tree.

\begin{definition}
We say that a set is \emph{everywhere nonmeager}, if it is nonmeager in every nonempty open set.
\end{definition}

\bt
\label{t:one-step}
Let $V$ be the ground model and suppose that $A \in V$ and $V \vDash A \s \N^\omega$ is everywhere nonmeager in $\N^\omega$. If $\mathbb{M}$ denotes the Miller forcing, $m$ denotes the Miller real, $C_m:= \prod\limits_{i \in \omega} [m(i)+1]$ and $G \subseteq \mathbb{M}$ is a generic filter, then $V[G] \vDash A \cap C_m$ is everywhere nonmeager in the nowhere dense Cantor set $C_m$.
\et

\begin{proof}
Suppose the contrary. Then there exists $s \in \N^{<\omega}$ such that $V[G] \vDash A \cap C_m \cap [s] \text{ is of first category in } C_m \cap [s] \neq \emptyset$, that is, there is a dense $G_\delta$ subspace of $C_m \cap [s]\neq \emptyset$ which is disjoint from $A \cap C_m \cap [s]$. In $V[G]$, let $U_1 \supseteq U_2 \supseteq ...$ be open sets in $\N^\omega$ such that for all $n$ we have $U_n \s B_\frac{1}{n}(C_m)$ (the $\frac{1}{n}$-neighborhood of the set $C_m$) and $U_n \cap C_m \cap [s]$ is dense in $C_m \cap [s] \neq \emptyset$ but $A \cap \bigcap\limits_{n \in \omega} U_n  \cap [s] = \emptyset$. Then, for the rest of the proof, we can fix $p^\star \in G$ such that 
$$
p^\star \Vdash  \dot{U_0} \supseteq \dot{U_1}... \text{ are open sets such that } \forall n \ \dot{U_n} \s \dot{B_\frac{1}{n}(C_m)},
$$
$$
\dot{U_n} \cap \dot{C_m} \cap [s] \text{ is dense in } \dot{C_m} \cap [s] \neq \emptyset, \text{ and } A \cap \bigcap\limits_{n \in \omega} \dot{U_n} \cap [s] = \emptyset.
$$

For every $p \le p^\star$ define (in $V$) the set
\begin{equation}\label{rpus}
R_p:=R_{p,\dot{\{U_n\}_{n \in \omega}}, s}:=\{r \in \N^\omega \mid \exists q \leq p: \; q \Vdash r \in \bigcap\limits_{n \in \omega} \dot{U_n} \cap [s] \}.
\end{equation}

In two difficult technical lemmas we will prove in the next subsections that $R_p \subseteq [s]$ is nonmeager for every $p \le p^\star$ (Lemma \ref{l:nonmeager}) and also that $R_p$ is analytic for dense many $p \leq p^\star$ (Lemma \ref{l:analytic}). Since every analytic set possesses the property of Baire, this implies that $R_p$ is comeager in a nonempty open subset of $[s]$ for dense many $p \leq p^\star$. Hence $A \cap R_p \neq \emptyset$ for dense many $p \leq p^\star$. Fix such a $p \leq p^\star$, and pick $r \in A \cap R_p$. Then $r \in A$ and there exists $q \le p$ such that $q \Vdash r \in  \bigcap\limits_{n \in \omega} \dot{U_n} \cap [s]$, hence $q \Vdash A \cap  \bigcap\limits_{n \in \omega} \dot{U_n} \cap [s] \neq \emptyset$. But $q \leq p^\star$, hence $q \Vdash A \cap  \bigcap\limits_{n \in \omega} \dot{U_n} \cap [s] = \emptyset$, a contradiction.
\end{proof}

\subsection{\texorpdfstring{$R_p$}{a} is nonmeager}

In this subsection we prove the following lemma promised above.

\begin{lemma}
\label{l:nonmeager}
$R_p$ is nonmeager for every $p \le p^\star$.
\end{lemma}

\begin{proof}
Let us start with two easy lemmas whose proofs are left to the reader.

\bl\label{ekviforszolás}
Given $\varrho \in \N^{<\omega}$ and $d \in \mathbb{M}$, the following are equivalent.
\begin{equation}\label{ekvi1}
  d \Vdash  \dot{C_m} \cap [\varrho]  \neq \emptyset,
\end{equation}
\begin{equation}\label{ekvi2}
   \forall i < \abs{\varrho} \ \ d \Vdash \varrho(i) \leq \dot{m}(i),
\end{equation}
\begin{equation}\label{ekvi3}
    \forall \sigma \in d \ \forall i < \min\{\abs{\varrho}, \abs{\sigma} \} \ \ \varrho(i) \leq \sigma(i).
\end{equation}
\el

\bl\label{rtovábbterjesztés}
Let $\dot{U}$ be a name for an open subset of $\N^\omega$, $\varrho \in \N^{<\omega}$ and $d \in \mathbb{M}$. If 
$$
d \Vdash \dot{U} \cap \dot{C_m} \cap [\varrho] \text{ is dense open in } \dot{C_m} \cap [\varrho] \neq \emptyset
$$ then there exists $d' \leq d$ and $\varrho' \succ \varrho$ such that
$$
d' \Vdash  \dot{C_m} \cap [\varrho'] \neq \emptyset \text{ and } [\varrho'] \s \dot{U}.
$$
\el

\bigskip

Let us now fix $p \le p^\star$ and prove that $R_p$ is nonmeager. We need to show that whenever $V_0 \supseteq V_1 \supseteq \ldots$ is a decreasing sequence of dense open sets in $\N^\omega$ then $\bigcap_n V_n  \cap R_p \neq \emptyset$. Fix such a sequence $\{V_n\}_{n \in \omega}$. On the one hand, we will construct an increasing sequence of  finite sequences $r_0 \prec r_1 \prec \ldots \in \N^{<\omega}$ such that $s \prec r_0$ and $[r_n] \s V_n$ for every $n$, and hence for  $r = \bigcup_n r_n$ we will have $r \in \bigcap_n V_n$. On the other hand, we will construct a decreasing sequence of Miller trees $p \geq p_0 \geq p_1 \geq....$ by a standard fusion argument in such a way that we will also have a Miller tree $p_\infty$ such that $p_n \geq p_\infty$ for every $n$ and $p_\infty$ will force that $r \in \bigcap_n \dot{U_n} \cap [s]$, showing that $r \in R_p$. Thus we will have $r \in \bigcap_n V_n  \cap R_p$, hence $\bigcap_n V_n  \cap R_p \neq \emptyset$, finishing the proof.

To build $p_\infty$ we will recursively choose vertices $t_0,t_1,... \in \N^{<\omega}$ that will be the branching nodes of $p_\infty$.

Now we carry out the recursive construction to obtain sequences $p_n, q_n \in \mathbb{M}, \;t_n, r_n \in \mathbb{N}^{<\omega}$. We require the following:
\begin{equation}\label{r0 kiterjeszti s-et}
r_0 \succ s,    
\end{equation}

\begin{equation}\label{rn része Vn}
    [r_n] \s V_n \quad \quad \forall n,
\end{equation}

\begin{equation}\label{rn-tl hossz}
    \abs{r_n} \geq \abs{t_l} \quad \quad \forall l \leq n,
\end{equation}

\begin{equation}\label{r Cm-beli}
    p_n \Vdash r_n(i) \leq \dot{m}(i) \quad \quad \forall i < \abs{r_n},
\end{equation}

\begin{equation}\label{branching}
t_l \in p_n \text{ is a branching node of } p_n \quad \quad \forall l \leq n,
\end{equation}

\begin{equation}\label{qn-pn relation}
    p_n[t_n] = q_n \quad \quad \forall n,
\end{equation}

\begin{equation}\label{qn forszolja rn része Un}
    q_n \Vdash [r_n] \s \dot{U_n} \quad \quad \forall n,
\end{equation}

\begin{equation}\label{rn increasing}
    \abs{r_{n+1}} > \abs{r_n}.
\end{equation}

Let $k:\omega \setminus \{0\} \longrightarrow \omega$ be a function such that $k(n+1) \leq n$ for all $n$ and every $l \in \omega$ comes up as a value of $k$ infinitely many times. This function will be a bookkeeping device, it will make sure that each $t_l$ is ``visited'' infinitely many times: if we have constructed $t_l$ for $l \leq n$, then the new branching node $t_{n+1}$ will be constructed such that $t_{k(n+1)} \prec t_{n+1}$. 

For the start of the recursion, take an extension  $r'_0 \succ s$ such that $[r'_0] \s V_0$. We thin out $p$ at heights $h \in \{ \abs{s},...,\abs{r'_0}-1 \}$ for values $r'_0(h)$. Call the resulting tree $p'_0$. We thinned out $p$ in such a way that \eqref{ekvi3} holds for $d:=p_0'$ and $\varrho:=r_0'$, so by Lemma \ref{ekviforszolás} we have that $p_0' \Vdash \dot{C_m} \cap [r_0'] \neq \emptyset$. Applying Lemma \ref{rtovábbterjesztés} there exists $r''_0$ and $q_0 \leq p_0'$ such that $r''_0 \succ r_0'$, $q_0 \Vdash [r''_0] \cap \dot{C_m} \neq \emptyset \text{ and } [r''_0] \s \dot{U_0}$. Finally, let $p_0:=q_0$, $t_0:=\stem(p_0)$ and to get $r_0$, extend $r_0''$ by some 0-s to get $\abs{r_0} \geq \abs{t_0}$. It is easy to verify that $\eqref{r0 kiterjeszti s-et}-\eqref{rn increasing}$ hold for $q_0,t_0,p_0,r_0$.

We move towards how to construct $q_{n+1},t_{n+1}, p_{n+1}, r_{n+1}$ if we have constructed $q_l,t_l,p_l, r_l$ for all $l \leq n$. First, we take an arbitrary extension $r_{n+1}'$ of $r_n$ such that $\abs{r_{n+1}'} > \abs{r_n}$ and $[r_{n+1}'] \s V_{n+1}$. We thin out $p_n$ at heights $h \in \{ \abs{r_n},...,\abs{r_{n+1}'}-1 \}$ for values $r_{n+1}'(h)$ to get a Miller tree $p_{n+1}' \leq p_n$. After the thinning out, $d:=p_{n+1}'$ and $\varrho:=r'_{n+1}$ satisfy $\eqref{ekvi3}$, so every extension of $p'_{n+1}$ will force that $\dot{C_m} \cap [r_{n+1}']  \neq \emptyset$.

The node $t_{k(n+1)}$ is a branching node of $p_{n+1}'$ because we thinned out $p_n$ at such heights, where $h \geq \abs{t_j} \; (j \leq n)$, so we can choose a $\sigma \in \rakovetkezo_{p_{n+1}'}(t_{k(n+1)})$ such that there is no $t_j$ for which $\sigma \prec t_j$ ($j \leq n$). Since $p_{n+1}' \Vdash [r_{n+1}'] \cap \dot{C_m} \neq \emptyset$, so $p_{n+1}'[\sigma] \Vdash [r_{n+1}'] \cap \dot{C_m} \neq \emptyset$, applying Lemma \ref{rtovábbterjesztés} there exists $r_{n+1}''\succ r_{n+1}' $ and $q_{n+1} \leq p_{n+1}'[\sigma]$ such that:
\begin{equation}\label{alámenő fák forszolnak}
    q_{n+1} \Vdash \emptyset \neq [r_{n+1}''] \cap \dot{C_m} \text{ and } [r_{n+1}''] \s \dot{U_{n+1}}.
\end{equation}

We thin out $p_{n+1}'$ for $r_{n+1}''$ (as we thinned out $p_n$ for $r_{n+1}'$), that is, thin out at heights $h \in \{ \abs{r_{n+1}'},...,\abs{r_{n+1}''}-1 \}$ for values $j_h=r_{n+1}''(h)$, call the resulting Miller tree $p_{n+1}''$. Note that \eqref{ekvi3} is satisfied by the pair $(p_{n+1}'',r_{n+1}'')$. We can now define $t_{n+1}$ and $p_{n+1}$:
$$
t_{n+1}:=\stem(q_{n+1}),
$$

$$
p_{n+1}:= (p_{n+1}'' \setminus p_{n+1}''[\sigma]) \cup q_{n+1}.
$$
To get $r_{n+1}$, we extend $r_{n+1}''$ with some 0-s to satisfy condition \eqref{rn-tl hossz}, that is, $\abs{r_{n+1}} \geq \abs{t_{n+1}}$. We also make the following trivial observation:
\begin{equation}\label{obeservation}
    p_{n+1}[t_{n+1}]=q_{n+1} \text{  and  } p_{n+1} \leq p_n.
\end{equation}
\bigskip
It is clear that conditions \eqref{r0 kiterjeszti s-et}, \eqref{rn része Vn}, \eqref{rn-tl hossz}, \eqref{qn-pn relation}, \eqref{qn forszolja rn része Un}, \eqref{rn increasing} are satisfied. 
Since whenever we thinned out a Miller tree in the induction step from $n$ to $n+1$, we thinned out at height $h \geq \abs{t_j} (j \leq n)$, so every $t_j$ will be a branching node of $p_{n+1}$ and $t_{n+1}$ was chosen to be a branching node of $p_{n+1}$, so \eqref{branching} also holds. 

To see that condition \eqref{r Cm-beli} holds for $(p_{n},r_n)$, by Lemma \ref{ekviforszolás} we only need to check that \eqref{ekvi3} holds. As we have already checked, condition \eqref{r Cm-beli} holds, when $n=0$. So let us check that condition \eqref{ekvi3} holds for $(p_{n+1},r_{n+1})$. 
Recall that we always thinned out $p_n,p_{n+1}'$ so that condition \eqref{ekvi3} holds for pairs $(p'_{n+1},r_{n+1}')$ and $(p_{n+1}'',r_{n+1}'')$. Since $p_{n+1} \leq p_{n+1}''$, \eqref{ekvi3} holds for the pair $(p_{n+1}, r_{n+1}'')$. In the construction of $r_{n+1}$ we have seen that $r_{n+1}$ was made by extending $r_{n+1}''$ by some 0-s, therefore \eqref{ekvi3} holds for the pair $(p_{n+1},r_{n+1})$, showing that condition \eqref{r Cm-beli} is also satisfied by our recursive construction. 

Let 
$$
r:= \bigcup\limits_{i \in \omega} r_i
$$ 
and
$$
q:= \{\varrho \in \N^{<\omega} \mid \exists n \in \omega: \; \varrho \prec t_n \}.
$$

We claim that $q \Vdash r \in \bigcap\dot{U_n}$. For this, it is enough to see that $q \Vdash r \in \dot{U_n}$ for all $n \in \omega$. Suppose for a contradiction that there exists an $n$ such that $q \nVdash r \in \dot{U_n}$. In this case there exists a $q' \leq q$ for which $q' \Vdash r \notin \dot{U_n}$. Since all the branching nodes of $q$ are $\{t_k \mid k \in \omega\}$ there is a $k > n$ such that $t_k$ is a branching node of $q'$. Since $q' \leq q \leq p_k$, by \eqref{obeservation} we have that $q'[t_k] \leq p_k[t_k]= q_{k}$ and by \eqref{alámenő fák forszolnak} we have that $q'[t_k] \Vdash [r_k] \s \dot{U_{k}}$, but $\dot{\{U_j\}_{j \in \omega}}$ is decreasing, hence $q'[t_k] \Vdash r \in \dot{U_{k}}$, which is a contradiction. Since $s \prec r_0 \prec r$, $q \Vdash r \in [s]$.

This concludes the proof of Lemma \ref{l:nonmeager}.
\end{proof}

\subsection{\texorpdfstring{$R_p$}{b} is analytic for a dense set of \texorpdfstring{$p \le p^\star$}{c}}

The goal of this subsection is to prove the following lemma promised above.

\begin{lemma}
\label{l:analytic}
The set $\{p \le p^\star: R_p \textrm{ is analytic} \}$ is dense below $p^\star$.
\end{lemma}

\bp
Recall that
$$
R_p = \{r \in \mathbb{N}^\omega \mid \exists q \leq p: \; q \Vdash r \in \bigcap\limits_{n \in \omega} \dot{U_n}  \cap [s] \},
$$
therefore, since $s \in \mathbb{N}^{<\omega}$ is in the ground model, 
$$
R_p = \{r \in \mathbb{N}^\omega \mid \exists q \leq p: \; q \Vdash r \in \bigcap\limits_{n \in \omega} \dot{U_n} \} \cap [s].
$$

Hence, it is enough to see that the following set is analytic for dense many $p \le p^\star$:
\begin{equation}\label{rsnélkül}
    Q_p:=\{r \in \mathbb{N}^\omega \mid \exists q \leq p: \; q \Vdash r \in \bigcap\limits_{n \in \omega} \dot{U_n} \}.
\end{equation}

We will need the following definition:
\bd
For a tree $T$ a subset $B \s T$ is called a \emph{barrier} if it intersects every infinite branch of $T$.
\ed

\bd
Let $p \le p^\star$ and $\dot{\{U_n\}_{n \in \omega}}$ be as above. Then we say that $p$ is \emph{nice with respect to} $\dot{\{U_n\}_{n \in \omega}}$ if for every $t \in \mathbb{N}^{< \omega} $ and every $n \in \omega$ the set
$$
B_{t,n}:=\{ \varrho \in p \mid p[\varrho] \Vdash [t] \s \dot{U_n} \text{ or } p[\varrho] \Vdash [t] \nsubseteq \dot{U_n} \}
$$
is a barrier in $p$. In such a situation we will say that $p[\varrho]$ decides $[t] \s \dot{U_n}$. 
\ed

We will prove on the one hand that the nice conditions are dense below $p^\star$, and on the other hand that $Q_p$ is analytic whenever $p$ is nice, which will complete the proof.

\bl
The set 
$$
\{p \leq p^\star \mid p \text{ is nice with respect to } \dot{\{U_n\}_{n \in \omega}} \}
$$
is dense below $p^\star$.
\el

\bp
Fix an arbitrary $p' \leq p^\star$ and let $\{(t_i, n_i)\}_{i \in \omega}$ be an enumeration of the elements of $\omega^{< \omega} \times \omega$.

We will inductively build a sequence of Miller trees $p_0 \geq p_1 \geq ...$ and sets $B_0, B_1, ...$ with the following properties:

\begin{equation}\label{hajjaj}
p_0 \leq p',
\end{equation}

\begin{equation}\label{Bi része pj}
\forall i,j \;  B_i \s p_j,
\end{equation}

\begin{equation}\label{Bi is a barrier in pi}
\forall i \; \  B_i \textrm{ is a barrier in } p_i,
\end{equation}

\begin{equation}\label{Bi-s are extended by some element from Bi+1}
    \forall \varrho \in B_i \; \exists^\infty \varrho' \in B_{i+1}:  \; \varrho \prec \varrho' \textrm{ and } \varrho \neq \varrho',
\end{equation}

\begin{equation}\label{Elements from Bi+1 are extensions of elements}
    \forall \varrho' \in B_{i+1} \; \exists \varrho \in B_i:\; \varrho \prec \varrho' \textrm{ and } \varrho \neq \varrho',
\end{equation}

\begin{equation}\label{Bi is an antichain}
    \forall i\textrm{ the elements of } B_i \textrm{ are incompatible,} 
\end{equation}

\begin{equation}\label{deciding}
\forall i \ \forall \varrho \in B_i \textrm{ the condition } p_i[\varrho] \textrm{ decides } [t_i] \s \dot{U_{n_i}}.
\end{equation}

Let us do the initial step of the induction. By the truth lemma, there is $p_0 \leq p'$ deciding $[t_0] \s \dot{U_{n_0}}$. Set $B_0:=\{\stem(p_0)\}$, then it is easy to check that \eqref{hajjaj}-\eqref{deciding} hold. 

Let us now carry out the general step of the induction. Suppose that $p_0 \geq p_1 \geq ... \geq p_i$ and $B_0, B_1, ... , B_i$ have been constructed with the required properties. For each $\sigma \in B_i$ and each $s_\sigma \in \rakovetkezo_{p_i}(\sigma)$ there exists $p_{s_{\sigma}} \leq p_i[s_\sigma]$ deciding $[t_{i+1}] \s \dot{U}_{n_{i+1}}$. Let
$$
p_{i+1}:= \bigcup\limits_{\substack{\sigma \in B_i \\ s_\sigma \in \rakovetkezo_{p_i}(\sigma)}} p_{s_{\sigma}},
$$
and
$$
B_{i+1}:= \left\{ \stem(p_{s_{\sigma}}) : \sigma \in B_i, s_\sigma \in \rakovetkezo_{p_i}(\sigma) \right\}.
$$

It is easy to check that \eqref{Bi része pj}-\eqref{Bi is an antichain} are satisfied. In order to verify \eqref{deciding}, note that if $\varrho \in B_{i+1}$ then $\varrho = \stem(p_{s_\sigma})$ for some $\sigma \in B_i$ and some successor $s_\sigma$, hence $p_{i+1}[\varrho] = p_{s_\sigma}$, therefore $p_{i+1}[\varrho]$ decides $[t_{i+1}] \s \dot{U}_{n_{i+1}}$.

Finally, we define $p$ as follows:
$$ 
p:=\{\tau \in \mathbb{N}^{<\omega} \vert  \exists i \in \omega \; \exists \varrho \in B_i: \tau \prec \varrho \}.
$$
By construction, $p$ is a Miller tree such that $p \leq p'$. One can also easily see that $p=\bigcap\limits_{i \in \omega} p_i$. Indeed, if $\tau \in p$, then for some $i \in \omega$ and $\varrho \in B_i$ we have $\tau \prec \varrho$, which implies that $\tau \in p_i$, but by \eqref{Bi-s are extended by some element from Bi+1} we know that $\varrho$ has an extension in $B_{i+1}, B_{i+2},...$, so we conclude that $\tau \in p_{i+1},p_{i+2},...$, so $p$ is contained in the intersection of $p_i$-s. To see the reverse containment, let us fix an arbitrary $\tau \in \bigcap\limits_{i \in \omega} p_i$ and say $length(\tau)=m$. If we take an arbitrary infinite branch $x$ of $p_{m+1}$ through $\tau$, then by \eqref{Bi is a barrier in pi} we know that there exists a $\varrho \in B_{m+1}$ such that $x\vert_{\lvert\varrho\rvert}=\varrho$. By \eqref{Elements from Bi+1 are extensions of elements} each element in $B_{m+1}$ has length at least $m+1$, therefore $\tau \prec \varrho$ and we conclude that $\tau \in p$.

We claim that $p$ is nice with respect to $\dot{\{U_n\}_{n \in \omega}}$. Indeed, if $(t,n) \in \omega^{<\omega} \times \omega$, then $(t,n) = (t_i,n_i)$ for some $i$. By \eqref{Bi is a barrier in pi} we know that $B_i$ is a barrier in $p_i$, moreover, $B_i \subseteq p$ by \eqref{Bi része pj}, therefore
$B_i$ is a barrier in $p$. Furthermore, by \eqref{deciding} we have that $\forall \varrho \in B_i$ the condition $p[\varrho] \textrm{ decides } [t_i] \s U_{n_i}$, hence the set  
$$B_{t_i,n_i}:=\{ \varrho \in p \mid p[\varrho] \textrm{ decides } [t_i] \s U_{n_i} \}
$$
contains $B_i$, which is a barrier in $p$, therefore $B_{t_i,n_i}$ is also a barrier in $p$.
\ep

We continue with a technical lemma which is essential in the proof of the analyticity of $Q_p$, as it will help us write $Q_p$ as a projection of a Borel set.

\bl
\label{F_n}
Let $p$ be nice with respect to $\dot{\{U_n\}_{n \in \omega}}$ and let $r \in \N^\omega$ be arbitrary. Then $r \in Q_p$ if and only if there exists a sequence of sets $F_n \s p$ such that the following hold for all $n \in \omega:$

\begin{equation}\label{kiindulás}
    \abs{F_0}=1,
\end{equation}

\begin{equation}\label{Fn elemei inkompatibilisek}
    \forall u,t \in F_n \ ( u \neq t \Rightarrow u \nprec t ),
\end{equation}

\begin{equation}\label{F_n elemeinek végtelen sok Fn+1 leszármazottja van}
    \forall u \in F_n \ \forall k \in \omega \ \exists l \geq k \ \exists t \in F_{n+1} \ u^\smallfrown l \prec t,
\end{equation}

\begin{equation}\label{Fn+1-ben közös elemeket kiterjesztők leghosszab közös kiterjesztése Fn beli}
    \forall u \in F_n  \ \forall t, t' \in F_{n+1} \ ( t,t' \succ u, t \neq t' \Rightarrow t \curlywedge t' = u ),
\end{equation}

\begin{equation}\label{Fn+1 elemei vmilyen Fn belit terjesztenek ki}
    \forall t \in F_{n+1} \  \exists u \in F_n \  t \succ u,
\end{equation}

\begin{equation}\label{még vmi label}
    \forall t \in F_n \  \exists k \in \omega  \ p[t] \Vdash [r\vert_k] \s \dot{U_n}.
\end{equation}

\el

\bp
First assume that there exists a sequence $\{F_n\}_{n \in \omega}$ of subsets of $p$ with the above properties. Our goal is to produce $q \leq p$ such that $q \Vdash r \in \bigcap\limits_{n \in \omega} \dot{U_n}$. Let 
$$
q:=\{ \sigma \mid \exists n \in \omega \ \exists t \in F_n \ \sigma \prec t \}.
$$

One can check from the construction that $q$ is a Miller tree,  $q \leq p$, and the branching nodes of $q$ are exactly the elements of $\bigcup\limits_{n \in \omega} F_n$.

To see that $q \Vdash r \in \bigcap\limits_{n \in \omega} \dot{U_n}$, suppose for a contradiction that there exists some $n \in \omega$ and $q' \leq q$ such that $q' \Vdash r \notin \dot{U_n}$. By the construction of $q$ it is clear that $F_n$ is a barrier in $q$, hence $F_n \cap q'$ is a barrier in $q'$. By assumption \eqref{még vmi label}, if we fix a $t \in F_n \cap q'$, then there exists $k \in \omega$ such that $p[t] \Vdash [r\vert_k] \s \dot{U_n}$. Then $q'[t] \leq p[t]$ and $q'[t] \leq q'$, hence $q'[t] \Vdash (r \in \dot{U_n}) \wedge (r \notin \dot{U_n})$, a contradiction.

For the other direction, suppose that we are given a $q \leq p$ such that $q \Vdash r \in \bigcap\limits_{n \in \omega} \dot{U_n}$. In order to build the appropriate $F_n$-s, we start with a claim:

\begin{claim}\label{vmilabel}
Let $p$ be nice with respect to $\dot{\{U_n\}_{n \in \omega}}$ and fix $n \in \omega$. Suppose that we are given a $p' \leq p$ such that $p'\Vdash r \in \dot{U_n}$, then there exists a branching node $t \in p'$  and $k \in \omega$ such that 
$$
p[t] \Vdash [r\vert_k] \s \dot{U_n}.
$$
\end{claim}

\bp By the truth lemma, there exists $p'' \leq p'$ and $k \in \omega$ such that $p'' \Vdash [r\vert_k] \s \dot{U_n}$. Since $p$ is nice with respect to $\dot{\{U_n\}_{n \in \omega}}$, $B_{r\vert_k,n}$ is a barrier. Therefore, if we choose $u \in p'' \cap B_{r\vert_k,n}$ then either $p[u] \Vdash [r\vert_k] \s \dot{U_n}$ or $p[u] \Vdash [r\vert_k] \nsubseteq \dot{U_n}$. Since $p''[u] \leq p[u]$ and $p''[u] \Vdash [r\vert_k] \s \dot{U_n}$, we conclude that $p[u] \Vdash [r\vert_k] \s \dot{U_n}$. Since $u$ is not necessarily a branching node, set $t:=\stem(p'[u])$, which clearly works.
\ep

Using Claim \ref{vmilabel} for $p':=q$ and for $n:=0$ we get a $t$ and a $k$ as above and we set $F_0:=\{t\}$. Then \eqref{kiindulás}, \eqref{Fn elemei inkompatibilisek} and \eqref{még vmi label} clearly hold for $n=0$ (the rest of the conditions are vacuous). 

To continue the construction of the sequence $F_1,F_2,...$ we need the following:

\begin{claim}\label{vmilabel2}
Let $p$ be nice with respect to $\dot{\{U_n\}_{n \in \omega}}$ and $q \leq p$. For every $n \in \omega$, branching node $\varrho  \in q$ and $\sigma \in \rakovetkezo_q(\varrho)$ there exists a branching node $t_\sigma$ of $q$ with $\sigma \prec t_\sigma$ and $k_\sigma \in \omega$ such that $p[t_\sigma] \Vdash [r\vert_{k_\sigma}] \s \dot{U_n}$.
\end{claim}

\bp Fix $n \in \omega$, a branching node $\varrho \in q$ and $\sigma \in \rakovetkezo_q(\varrho)$. Apply Claim \ref{vmilabel} for $n$ and $p':=q[\sigma]$ to obtain a branching node $t_\sigma \in q[\sigma]$ and $k_\sigma \in \omega$ such that $p[t_\sigma] \Vdash [r\vert_{k_\sigma}] \s \dot{U_n}$.
\ep

Let us return to the construction of the sequence $\{F_n\}_{n \in \omega}$. Suppose that we have constructed $F_0,...,F_n$ such that \eqref{kiindulás}-\eqref{még vmi label} hold and suppose also that every element of $F_n$ is a branching node of $q$. Apply the previous claim for $n+1$, every $\varrho \in F_n$, and every $\sigma \in \rakovetkezo_q(\varrho)$. Thus we obtain a branching node $t_\sigma$ and $k_\sigma$ such that
$p[t_\sigma] \Vdash [r\vert{_{k_\sigma}}] \s \dot{U_{n+1}}$. Let
$$
F_{n+1}:= \left\{ t_\sigma : \varrho \in F_n, \ \sigma \in \rakovetkezo_q(\varrho) \right\}.
$$
Then it is clear that the sequence $\{F_n\}_{n \in \omega}$ satisfies \eqref{kiindulás}-\eqref{még vmi label}. This concludes the proof of Lemma \ref{F_n}.
\ep

As an immediate corollary to Lemma \ref{F_n} we obtain the following.

\begin{claim}\label{közel a vége}
Let $p$ be nice with respect to $\dot{\{U_n\}_{n \in \omega}}$, then $Q_p$ is the projection on the first coordinate of the following set
$$
B_p:=\{(r,\{F_n\}_{n \in \omega}) \mid (r,\{F_n\}_{n \in \omega}) \text{ satisfies conditions } \eqref{kiindulás}-\eqref{még vmi label}\}.
$$
\end{claim}

Since $F_n \s p$, we can view $F_n$ as an element of $Y:=2^p$ and we can view $\{F_n\}_{n \in \omega}$ as an element of $Y^\omega$. Clearly, $Y$ is Polish (in fact, compact metrizable), hence so is $Y^\omega$. 

\bl\label{nagyon közel}
Let $p$ be nice with respect to $\dot{\{U_n\}_{n \in \omega}}$, then $B_p \s \mathbb{N}^\omega \times Y^\omega$ is Borel.
\el

\bp
It is enough to check that each of the following sets is Borel:
$$
\{(r,\{F_n\}_{n \in \omega}) \mid (r,\{F_n\}_{n \in \omega}) \text{ satisfies condition } \eqref{kiindulás}\},
$$
$$
\vdots
$$
$$
\{(r,\{F_n\}_{n \in \omega}) \mid (r,\{F_n\}_{n \in \omega}) \text{ satisfies condition } \eqref{még vmi label}\}.
$$

We mostly leave this standard but very lengthy computation to
the reader, and only deal with the most interesting clause, that is, when condition $\eqref{még vmi label}$ holds.

Clearly,
$$\{(r,\{F_n\}_{n \in \omega}) \mid (r,\{F_n\}_{n \in \omega}) \text{ satisfies condition } \eqref{még vmi label}\}=$$
$$= \bigcap\limits_{n \in \omega} \bigcap\limits_{t \in p} \bigcup\limits_{k \in \omega} \{ (r,\{F_n\}_{n \in \omega}) \mid (t \notin F_n) \vee (r \in V_{n,t,k}) \},$$
where
$$
V_{n,t,k}:=\{r \in \mathbb{N}^\omega \mid p[t] \Vdash[r\vert_k] \s \dot{U_n}\}.
$$

So it is enough to check that each $V_{n,t,k}$ is Borel, but this set is clearly open, because if $r \in V_{n,t,k}$ then the neighbourhood $[r\vert_k] \s V_{n,t,k}$.
\ep

\begin{lemma}
Let $p$ be nice with respect to $\dot{\{U_n\}_{n \in \omega}}$, then the set $Q_p$ is analytic.
\end{lemma}

\bp
By Lemma \ref{nagyon közel} the set $B_p$ is Borel and by Claim \ref{közel a vége} the set $Q_p$ is the projection of $B_p$.
\ep

This completes the proof of Lemma \ref{l:analytic}.
\ep

\subsection{The iteration}

Now we are ready to prove the main result of the section. Recall that the \emph{Miller model} is the model obtained by the countable support iteration of length $\omega_2$ of the Miller forcing $\mathbb{M}$ over a model of the Continuum Hypothesis.

\begin{theorem}
\label{t:nonmeager}
The following statement is consistent, since it holds in the Miller model:
If $X$ is a perfect Polish space, and $\mathcal{A}$ is a family of subsets of $X$ such that $|\mathcal{A}| \le \omega_1$ and $A$ is everywhere nonmeager in $X$ for every $A \in \mathcal{A}$, then there exists a nowhere dense Cantor set $C \subseteq X$ such that $A \cap C$ is everywhere nonmeager in $C$ for every $A \in \mathcal{A}$.
\end{theorem}

\begin{proof}
Let $V_{\omega_2}$ be the Miller model, and let $V_\alpha$ denote the intermediate models of the iteration.

The following argument takes place in $V_{\omega_2}$: We claim that we may assume that $X = \N^\omega$. Indeed, 
by Lemma \ref{perfectcontbairelemma}, $X$ contains a dense $G_\delta$ subspace $X'$ homeomorphic to $\N^\omega$, and clearly $A \cap X'$ is everywhere nonmeager in $X'$ for every $A \in \mathcal{A}$, hence the resulting Cantor set $C \subseteq X'$ works.


In $V_{\omega_2}$, let  $\mathcal{A} = \{ A^\xi : \xi < \omega_1 \}$ be a family of subsets of $\N^\omega$ such that $A^\xi$ is everywhere nonmeager in $\N^\omega$ for every $\xi < \omega_1$. In fact, we identify $\mathcal{A}$ and $A^\xi$ with their $\M_{\omega_2}$-names. We may assume that for every $\xi$ the set $A^\xi$ is enumerated as $A^\xi = \{ x^\xi_\beta \}_\beta$, and every $x^\xi_\beta$ is a so called nice name, that is, it consists of tuples $((n, k), p)$ determining whether the relation $x^\xi_\beta(n) = k$ holds, such that for every fixed $n$ the corresponding $p \in \M_{\omega_2}$ determining the value $x^\xi_\beta(n)$ form an antichain.

For every $\xi$ and $\alpha < \omega_2$ let $A^\xi_\alpha$ be the name consisting of the $\{ x^\xi_\beta \}_\beta$ for which the support of all the corresponding $p$ are below $\alpha$. Then clearly $A^\xi_\alpha$ is an $\M_\alpha$-name, and hence the corresponding set $A^\xi_\alpha \in V_\alpha$.

\begin{lemma}
There exists $\alpha^\star < \omega_2$ such that for every $\xi < \omega_1$ we have $V_{\alpha^\star} \vDash A^\xi_{\alpha^\star}$ is everywhere nonmeager in $\N^\omega$.
\end{lemma}

\begin{proof}
Using that $CH$ holds in the intermediate models, hence there are only $\omega_1$ many Borel sets coded in each intermediate model, and also using that by the $\omega_2$-chain condition every $x^\xi_\beta$ shows up in an intermediate model, we can define a strictly increasing continuous sequence $\langle \alpha_\gamma : \gamma < \omega_1 \rangle$ of ordinals less than $\omega_2$ such that for every $\gamma < \omega_1, \xi < \omega_1$ and $s \in \N^{< \omega}$
\[
V_{\omega_2} \vDash A^\xi_{\alpha_{\gamma + 1}} \cap [s] \textrm{ is not covered by any meager Borel set coded in } V_{\alpha_\gamma}.
\]

But then setting $\alpha^\star = \sup_\gamma \alpha_\gamma$ works, since by properness every meager Borel set coded in $V_{\alpha^\star}$ already shows up in a $V_{\alpha_\gamma}$.
\end{proof}

Now we complete the proof of the theorem. From the above lemma and Theorem \ref{t:one-step} we obtain that
\[
V_{\alpha^\star + 1} \vDash A^\xi_{\alpha^\star} \cap C_{m_{\alpha^\star + 1}} \textrm{ is everywhere nonmeager in } C_{m_{\alpha^\star + 1}} \textrm{ for every } \xi.
\]
Since Miller forcing is category-preserving (see \cite[Theorem 7.3.46, Lemma 6.3.19 and Theorem 6.3.20]{Bart}), this implies that
\[
V_{\omega_2} \vDash A^\xi_{\alpha^\star} \cap C_{m_{\alpha^\star + 1}} \textrm{ is everywhere nonmeager in } C_{m_{\alpha^\star + 1}} \textrm{ for every } \xi.
\]
But then $A^\xi_{\alpha^\star} \subseteq A^\xi$ implies
\[
V_{\omega_2} \vDash A^\xi \cap C_{m_{\alpha^\star + 1}} \textrm{ is everywhere nonmeager in } C_{m_{\alpha^\star + 1}} \textrm{ for every } \xi,
\]
concluding the proof.
\end{proof}

\section{Reflection of non-Baire property}

In this section we use the main result of the previous section to prove that consistently non-Baire property reflects in nowhere dense Cantor subspaces. This result will be the key ingredient in the next section, where the consistent equivalence of two definitions of universal Baireness will be proved.

\bt
\label{thm:non-Baire}
The following statement is consistent with $ZFC$, since it holds in the Miller model: If $X$ is a Polish space and a subset $A \subseteq X$ does not have the property of Baire, then there exists a nowhere dense Cantor set $C \subseteq X$ such that $A \cap C$ does not have the property of Baire in $C$.
\et

\bp
The following argument takes place in the final model: We claim that we may assume that $X = \N^\omega$. Let $X$ be an arbitrary Polish space, and let a set $A \subseteq X$ without the property of Baire be given.
It is well-known that we can write $X$ as a disjoint union $X = X_\textsl{perf} \cup X_\textsl{ctbl}$,  where $X_\textsl{perf}$ is perfect (closed without isolated points) and $X_\textsl{ctbl}$ is countable. We claim that $A \cap \inter X_\textsl{perf}$ does not have the property of Baire (in particular, $\inter X_\textsl{perf}$ is nonempty). Assume otherwise. Since $X_\textsl{perf} \setminus \inter X_\textsl{perf}$ is nowhere dense and $X_\textsl{ctbl}$ is countable, every subset of the latter two sets has the property of Baire, therefore we obtain that $A$ also has the property of Baire, a contradiction. Therefore $A \cap \inter X_\textsl{perf}$ also does not have the property of Baire in $\inter X_\textsl{perf}$ by Lemma \ref{bairetultranz}. Clearly, $\inter X_\textsl{perf}$ itself is a perfect Polish space, hence by Lemma \ref{perfectcontbairelemma} it contains a dense $G_\delta$ subspace $X'$ homeomorphic to $\N^\omega$, and similarly as above, $A \cap X'$ does not have the property of Baire in $X'$. Then, using the case $X = \N^\omega$ we obtain a nowhere dense Cantor set $C \subseteq X'$ such that $A \cap C$ does not have the property of Baire in $C$, and we are done.

Let now a set $A \subseteq \N^\omega$ without the property of Baire be given. Define
\[
\mathcal{U} = \{ [s] : s \in \N^{<\omega} \textrm{ and } \left( [s] \cap A \textrm{ is meager or } [s] \setminus A \textrm{ is meager} \right)\}.
\]
It is well-known and easy to see that $A$ has the property of Baire if and only if $\bigcup \mathcal{U}$ is dense, therefore there exists an $s_0 \in \N^{<\omega}$ such that $[s_0] \cap \bigcup \mathcal{U} = \emptyset$. But then $[s_0] \cap A$ and $[s_0] \setminus A$ are both everywhere nonmeager in the perfect Polish space $[s_0]$. Applying Theorem \ref{t:nonmeager} for this Polish space with $\mathcal{A} = \{ [s_0] \cap A, [s_0] \setminus A \}$ we obtain a nowhere dense Cantor set $C \subseteq [s_0]$ such that $C \cap A$ and $C \setminus A$ are both everywhere nonmeager in $C$, which easily implies that $A \cap C$ does not have the property of Baire in $C$. 
\ep

We also remark that certain analogous questions  have been considered by various authors. A set (in the Baire space) is called compactly $\mathbf{\Sigma^1_n}$ if its intersection with every compact set is $\mathbf{\Sigma^1_n}$ (for a definition of $\mathbf{\Sigma^1_n}$ sets, see \cite[Section 37]{Ke}). In \cite{KM} and \cite{Be} K. Kunen-A. W. Miller and H. Becker investigated the consistency of the statements that a set is compactly $\mathbf{\Sigma^1_n}$ if and only if it is $\mathbf{\Sigma^1_n}$.

Using a similar terminology the result of this section could essentially be stated as: consistently, if a set is compactly BP then it is BP.

\customsection{Polish and compact metrizable domains are consistently equivalent}{Polish and compact metrizable domains are consistently equivalent}{Polish vs. compact metrizable domains}

In this section we use the main result of the previous section to show that two definitions of universal Baireness are consistently equivalent.

\begin{theorem}
\label{t:equiv}
It is consistent with $ZFC$, since it holds in the Miller model, that a subset of a Polish space is Polish-universally Baire if and only if it is compact metrizable-universally Baire.
\end{theorem}

\bp
Every Polish-universally Baire set is clearly compact metrizable-univer\-sally Baire, so we only need to prove the other implication. Let $X$ be a Polish space and let $A \subseteq X$ be compact metrizable-universally Baire.
Let $Y$ be another Polish space, and let $f : Y \to X$ be continuous. We need to check that $f^{-1}(A)  \subseteq Y$ has the property of Baire. Assume otherwise. Applying Theorem \ref{thm:non-Baire} in $Y$ we obtain a nowhere dense Cantor set $C \subseteq Y$ such that $f^{-1}(A) \cap C$ does not have the property of Baire in $C$. But then for the continuous map $f|_C : C \to X$ defined on the compact metrizable space $C$ we have that $\left( f|_C \right)^{-1} (A) = f^{-1}(A) \cap C$ does not have the property of Baire in $C$, a contradiction.
\ep

\section{Relation to sets with the property of Baire in the restricted sense}
\label{s:restricted}

Recall that a set $A \s X$ has the \underline{property of Baire in the restricted sense} if $A \cap P$ has the property of Baire in $P$ for every perfect $P \s X$. Since every perfect $P \s X$ is Polish, we have the following fact:

\begin{fact}\label{Polish imply restricted}
    Every Polish-universally Baire $A \s X$ has the property of Baire in the restricted sense.
\end{fact}

The main goal of this section is to show that it is independent from $ZFC$ that Polish-universally Baire sets coincide with sets with the property of Baire in the restricted sense.

Recall that a set $A \s X$ is called \underline{perfectly meager} if $A \cap P$ is meager in $P$ for every perfect $P \s X$. The following theorem shows that consistently the property of Baire does not imply Polish-universal Baireness.

\begin{theorem}\label{pm nem um}
    Under CH there exists a set $A \s [0,1]$ such that $A$ has the property of Baire in the restricted sense, but $A$ is not Polish-universally Baire (and by Theorem \ref{thm:diagram2}, the set $A$ is neither compact metrizable-universally Baire).
\end{theorem}

\begin{proof}
    In \cite{Sierp} it was proved that under CH there exists a set $A' \s [0,1]$ such that $A'$ is perfectly meager, but not universally meager. It is clear that $A'$ and each of its subsets have the property of Baire in the restricted sense. 
    
    If $A'$ is not Polish-universally Baire, then set $A:=A'$ and we are done. If $A'$ is Polish-universally Baire, then using that $A'$ is not universally meager we obtain a Borel isomorphism $f:[0,1] \longrightarrow [0,1]$ such that $f^{-1}(A')$ has the property of Baire, but not meager. Since any nonmeager subset of $[0,1]$ that has the property of Baire contains a set without the property of Baire, select any $A \s A'$ such that $f^{-1}(A)$ does not have the property of Baire. By Theorem \ref{thm: boole equivalence} \ref{Borel isomorphic preimage} the set $A$ is not Polish-universally Baire.
\end{proof}

Now we will show that consistently the property of Baire in the restricted sense implies Polish-universal Baireness.

\begin{theorem}\label{restricted Baire implies Polish}
It is consistent with $ZFC$, since it holds in the Miller model, that $A \s X$ is Polish-universally Baire if and only if $A$ is compact metrizable-universally Baire if and only if $A$ has the property of Baire in the restricted sense.
\end{theorem}

\begin{proof}
    By Theorem \ref{t:equiv} we have that in the Miller model $A$ is Polish-universally Baire if and only if $A$ is compact metrizable-universally Baire and Fact \ref{Polish imply restricted} shows that if $A$ is Polish-universally Baire then $A$ has the property of Baire in the restricted sense.

    Now, suppose that $A$ has the property of Baire in the restricted sense. Suppose for a contradiction that $A$ is not Polish-universally Baire. By Theorem \ref{thm: boole equivalence} \ref{Polish topology refinement} there exists a Polish topology $\tau'$ on $X$ that refines the original topology $\tau$ of $X$ and $A$ does not have the property of Baire in $(X,\tau') $. By Theorem \ref{thm:non-Baire} there exists a nowhere dense Cantor set $C$ in $(X,\tau')$ such that $A \cap C$ does not have the property of Baire in $(C, \tau'\vert_C)$. Finally, since $C$ is a Cantor set we have that $\tau'\vert_C = \tau\vert_C$ and this contradicts the fact that $A \cap C$ has the property of Baire in the perfect set $(C, \tau\vert_C)$.
\end{proof}

In Theorem \ref{pm nem um} we have seen a consistent example that shows that the property of Baire in the restricted sense does not imply compact metrizable-universal Baireness. Let us remark that a Luzin set (see Example \ref{ex:L}) is a consistent example of a set that is compact metrizable-universally Baire, but does not have the property of Baire and hence does not have the property of Baire in the restricted sense.

\section{Arguments for one of the definitions}
\label{s:argument}

What is the natural Baire category dual to the notion of a universally measurable set? (Recall that a set in a Polish space is called \emph{universally measurable} if it is measurable with respect to the completion of every Borel probability measure.) In this section we argue that the best fit is the class of Polish-universally Baire sets.

\medskip

\noindent\underline{First Argument:}

\medskip

Let us try to axiomatize what we expect from the Baire category dual to the notion of universally measurable set. The following list of requirements contains the duals of the most basic properties of universally measurable (and universally null) sets.

\begin{enumerate}[label=(\roman*)]
    \item\label{Borelekub-k} Every Borel set is universally Baire.
    \item\label{zártsouslinra} Universally Baire sets are closed under Souslin operation.
    \item\label{zártfolytősképekre} The universally Baire sets are closed under continuous preimages between Polish spaces.
    \item\label{univbaire-k BP-k} Every universally Baire set has the property of Baire.
    \item\label{univmeagerunivbaire} Every universally meager set is universally Baire.
\end{enumerate}

Well, properties \ref{Borelekub-k}-\ref{zártfolytősképekre} are satisfied by $\mathcal{F}$-universally Baire sets for an arbitrary $\mathcal{F}$.

Let us consider now property \ref{univbaire-k BP-k}. We saw in Example \ref{ex:L} that a Luzin set in $\N^\omega$ is a set without the property of Baire that is both compact-universally Baire and compact metrizable-universally Baire.  This already rules out compact-universal Baireness and compact metrizable-universal Baireness to be the natural dual.

Property \ref{univmeagerunivbaire} rules out topological-universally Baire sets, because in Example \ref{ex:twowoodincardinal} under certain large cardinal assumptions we give an example of a universally meager set that is not topological-universally Baire.

\begin{remark}
    Let us remark here that we use the definition of a universally meager set in the sense of Zakrzewski \cite{Zak1}. However, there is another very useful universally meager notion due to Todor\v{c}evi\'c. Let us remark here that consistently every universally meager set - in the sense of Todor\v{c}evi\'c - is countable, but in the measure case we know that there exists an uncountable universally null set.
\end{remark}

This leaves us with the only option of Polish-universal Baireness.

\medskip

We now would like to argue further that this is indeed the natural choice.

\medskip

\underline{Second argument:}

\medskip

Recall the following important characterization of universally measurable sets:

\bt\label{univmeas} (see \cite[434D]{Fremlin4}) Let $X$ be a Polish space with $A \s X$. The set $A$ is universally measurable if and only if $f^{-1}(A)$ is Lebesgue measurable for every Borel function $f:[0,1] \longrightarrow X$.
\et

The next result strongly supports our choice.

\begin{theorem}\label{truth}
Let $X$ be a Polish space with $A \s X$. The set $A$ is Polish-universally Baire if and only if $f^{-1}(A)$ has the property of Baire for every Borel function $f: [0,1] \longrightarrow X$.
\end{theorem}

\bp
If $A$ is Polish-universally Baire then Claim \ref{cor: a jó választás} yields that $f^{-1}(A)$ has the property of Baire for every Borel function $f: [0,1] \longrightarrow X$.

Conversely, suppose that $f^{-1}(A)$ has the property of Baire for every Borel function $f: [0,1] \longrightarrow X$. By Theorem \ref{thm:diagram1} it is enough to prove that $A$ is $\{ \N^\omega \}$-universally Baire. Let $g: \bs \longrightarrow X$ be continuous, we need to show that $g^{-1}(A)$ has the property of Baire. It is well-known that $\bs$ is homeomorphic to $[0,1] \setminus \mathbb{Q}$ and throughout the proof we will identify $\bs$ with $[0,1] \setminus \mathbb{Q}$.

If $A = X$ then $A$ is clearly $\{ \N^\omega \}$-universally Baire, so we can assume that $A \neq X$ and fix $x_0 \in X \setminus A$. Define
\begin{equation*}
    h(a)=
    \begin{cases*}
    g(a) & if $a \in [0,1] \setminus \mathbb{Q}$, \\
    x_0 & if $a \in [0,1] \cap \mathbb{Q}$.
    \end{cases*}
\end{equation*}

It is clear that $h : [0,1] \to X$ is a Borel function, hence $h^{-1}(A)=g^{-1}(A)$ has the property of Baire in $[0,1]$. To see that $g^{-1}(A)$ has the property of Baire in $\bs$, we choose an open set $U$ in $[0,1]$ and a meager set $M$ in $[0,1]$ such that $U \triangle M= g^{-1}(A)$. It is easy to see that $g^{-1}(A)= (U \cap \bs) \triangle (M \cap \bs) $ and $(U \cap \bs)$ is open in $\bs$ and $ (M \cap \bs)$ is meager in $\bs$, hence $g^{-1}(A)$ has the property of Baire in $\bs$.
\ep

\underline{Third argument:}

\medskip

As we have seen in the first argument, compact-universally and compact metrizable-universally Baire are not the appropriate notions for universally Baire sets, since they consistently do not even imply the property of Baire. In the following we argue that the theory of topological-universally Baire sets is not sufficiently parallel to that of universally measurable sets.

Interestingly, it is well-known that topological-universal Baireness implies measurability. But universal measurability does not imply the property of Baire (in $ZFC$). Finally, we should also add the point that the rich theory of topological-universal Baireness simply has no counterpart for universal measurability.

\medskip

\medskip
To sum up, it would perhaps be more fortunate to use the term 'universally Baire' for Polish-universally Baire. The only source that we are aware of to do this so far is \cite{Matheron}, but perhaps this practice should be followed in descriptive set theory texts from now on.

\section{Concluding remarks}\label{s:concluding}


In the previous section we argued that the natural category dual notion to universal measurability is Polish-universally Baire. Let us give a further dual behaviour between Polish-universally Baire sets and universally measurable sets. Recall that if $\mathcal{A}$ is a $\sigma$-algebra on $X$, then we can define the $\sigma$-ideal of small sets of $\mathcal{A}$ as:
$$
\mathcal{I}_\mathcal{A}:= \{H \s X \mid H' \in \mathcal{A} \text{ for every } H' \s H\}.
$$

For example, if $X=\R$ and $\mathcal{A}=$ Lebesgue measurable sets/sets with the property of Baire/Borel sets, then $\mathcal{I}_\mathcal{A}=$ Lebesgue null sets/meager sets/countable sets. It is also well-known that if $\mathcal{A}=$ universally measurable subsets of $X$, then $\mathcal{I}_\mathcal{A}$= universally null sets of $X$. Let us show that the small sets of Polish-universally Baire sets are exactly the universally meager sets:

\begin{theorem}
    Let $\mathcal{A}$ denote the $\sigma$-algebra of Polish-universally Baire subsets of a perfect Polish space $X$, then $\mathcal{I}_\mathcal{A}=$ universally meager sets.
\end{theorem}

\begin{proof}
    It is clear that every universally meager subset of $X$ is in $\mathcal{I}_\mathcal{A}$, since every universally meager set is Polish-universally Baire.

    Now assume $A \in \mathcal{I}_\mathcal{A}$, that is, $A$ and each of its subsets are Polish-universally Baire sets and the goal is to conclude that $A$ is universally meager. Suppose for a contradiction that $A$ is not universally meager. In this case, there exists a Borel isomorphism $f:X \longrightarrow X$ such that $f^{-1}(A)$ is not meager in $X$. Since $f^{-1}(A)$ has the property of Baire, but not meager, it has a subset $B \s f^{-1}(A)$ such that $B$ does not have the property of Baire. If $A':= f(B) \s A$, then $A'$ is not Polish-universally Baire, a contradiction.
\end{proof}

Throughout the paper we have seen many definitions of universally Baire sets. Another natural attempt could be the following: define a set $A$ to be universally Baire if $A=B \vartriangle C$, where $B$ is a Borel set and $C$ is a universally meager set. In \cite[Proposition (ii)]{Sol} Solecki shows that there exists an analytic set $A$ such that $A \vartriangle B$ contains a nonempty perfect set for every coanalytic set $B$. Since a universally meager set could not contain a nonempty perfect set, we conclude that the family of sets $\{B \vartriangle C \mid B \text{ is Borel and }C \text{ is universally meager}\}$ does not include every analytic set. This shows that this notion is not a satisfactory candidate.

Let us define the additivity of Polish-universally Baire sets as:
\begin{definition}
We define the additivity of Polish-universally Baire sets as   
    $$
    \hspace{-4pt} \add(\text{Polish-uB}):= \min\{\kappa \mid \exists \{B_\alpha\}_{\alpha < \kappa} \s \bs\text{ s.t. } \forall\alpha \; B_\alpha \text{ is Polish-universally }
    $$
    $$
     \text{Baire and } \bigcup\limits_{\alpha < \kappa} B_\alpha \text{ is not } \text{Polish-universally Baire}\}.
    $$
\end{definition}
Notice that we can replace $\bs$ with any other uncountable Polish space in the previous definition, since by Theorem \ref{thm: boole equivalence} \ref{Borel isomorphic preimage} Polish-universally Baire sets are closed under Borel isomorphisms.
\begin{theorem}
    $\add(\text{Polish-uB}) = \add(\mathcal{M})$.
\end{theorem}

\begin{proof}
    First, we prove that $\add(\text{Polish-uB}) \geq \add(\mathcal{M})$. Consider an arbitrary $\kappa < \add(\mathcal{M})$ and we show that the union of  $\kappa$ many Polish-universally Baire sets $\{B_\alpha\}_{\alpha < \kappa}$ in a Polish space $X$ is still Polish-universally Baire. Let $f$ be a continuous function $Y \longrightarrow X$ from an arbitrary Polish space $Y$. Using the well-known fact that the union of less than $\add(\mathcal{M})$ many sets with the property of Baire has the property of Baire, we conclude that $f^{-1}(\bigcup\limits_{\alpha < \kappa} B_\alpha)=\bigcup\limits_{\alpha < \kappa} f^{-1}(B_\alpha)$ has the property of Baire.

    Now we will construct $\add(\mathcal{M})$ many Polish-universally Baire sets such that their union is not Polish-universally Baire, in fact, their union will not have the property of Baire. For brevity, let $\kappa:=\add(\mathcal{M})$ and by the definition of $\kappa$ there exists meager sets $\{M_\alpha\}_{\alpha < \kappa}$ such that $\bigcup\limits_{\alpha < \kappa} M_\alpha$ is not meager. Let $B_0$ be a Borel meager set that covers $M_0$ and by transfinite recursion on $\alpha$ let $B_\alpha$ be a Borel meager set that covers the meager set $\big(\bigcup\limits_{\gamma \leq \alpha} M_\gamma \big) \cup \big(\bigcup\limits_{\delta < \alpha} B_\delta\big)$. Notice that $\bigcup\limits_{\alpha < \kappa} B_\alpha$ is not meager. 

    For every $\alpha < \kappa$ let $C_\alpha:= B_\alpha \setminus \bigcup\limits_{\gamma < \alpha} B_\gamma$. Every $C_\alpha$ is Polish-universally Baire, since they are of the form of a Borel set minus the union of less than $\kappa$ many Borel sets. By the Four Poles Theorem \cite{4P} there exists $A \s \kappa$ such that $\bigcup\limits_{\alpha \in A} C_\alpha$ does not have the property of Baire.
\end{proof}

\section{Open problems}
\label{s:open}

The most relevant open questions are whether the consistent examples demonstrating that the various definitions of universal Baireness are consistently nonequivalent can be replaced by $ZFC$ examples. Of course, Theorem \ref{t:equiv} settles this question in the case of Polish-universally Baire and compact metrizable-universally Baire, so the following six questions remain open.

\begin{question}
Is there, in $ZFC$, a Polish-universally Baire set in a Polish space that is not topological-universally Baire?
\end{question}

\begin{question}
Is there, in $ZFC$, a compact-universally Baire set in a Polish space that is not topological-universally Baire?
\end{question}

\begin{question}
Is there, in $ZFC$, a compact metrizable-universally Baire set in a Polish space that is not compact-universally Baire?
\end{question}

\begin{question}
Is there, in $ZFC$, a Polish-universally Baire set in a Polish space that is not compact-universally Baire?
\end{question}

\begin{question}
Is there, in $ZFC$, a compact-universally Baire set in a Polish space that is not Polish-universally Baire?
\end{question}

\begin{question}
Is there, in $ZFC$, a compact metrizable-universally Baire set in a Polish space that is not topological-universally Baire?
\end{question}

Two consistent examples above, namely Examples \ref{ex:S} and \ref{ex:3} both utilize the well-known fact that topological-universally Baire sets are measurable. This motivates the next question.

\begin{question}
Is there, in $ZFC$, a Polish-universally Baire set in $[0,1]$ that is not Lebesgue measurable?
\end{question}

Note that this is exactly the dual of the famous open problem whether consistently every universally measurable set has the  property of Baire, see \cite{MPW}.

\bigskip
Next we pose two problems concerning $\compact$ sets.

\begin{question}
    Is it true in ZFC that $\compact$ sets are not closed under taking Borel preimages?
\end{question}

\begin{question}
    Is a set in a Polish space $\compact$ iff it is $\{2^\kappa \mid \kappa \textrm{ is a cardinal} \}$-universally Baire? 
\end{question}

\bigskip

Our next set of questions is slightly vague. As we have argued in Section \ref{s:argument} above, the natural dual to universally measurable is Polish-universally Baire, while requiring topological-universal Baireness is much stronger.


\begin{question}
What is the measure theoretic dual to topological-universally Baire?
\end{question}

One natural guess for the above question would be of the following form: every continuous/Borel measurable preimage of $A \s X$ in any topological space $Y$ is measurable with respect to the completion of every Borel probability measure. One can easily see that the previous definition coincides with the universally measurable sets. If $f:Y \longrightarrow X$ is an arbitrary Borel function from an arbitrary topological space and $\mu$ is a Borel probability measure on $Y$, then let $\nu$ be the pushforward measure of $\mu$. Then there exist Borel sets $C,D$ such that $C \s A \s D$ and $\nu(D \setminus C)=0$. The latter implies that $f^{-1}(A)$ is measurable with respect to $\mu$ because $f^{-1}(C) \s f^{-1}(A) \s f^{-1}(D)$, the sets $f^{-1}(C), f^{-1}(D)$ are Borel and $\mu(f^{-1}(D) \setminus f^{-1}(C))=0$.

An interesting test question would be whether every set possessing this dual property has the property of Baire. A further intriguing question is whether this new dual notion could actually be equivalent to topological-universal Baireness.

\bigskip

Throughout the paper, the underlying space has always been a Polish space. 

\begin{question}
What can be said about universally Baire sets in more general topological spaces?
\end{question}

The best reference for results of this kind is \cite{Fremlin}.

\section*{Acknowledgements} 

We are grateful to Balázs Bursics, Qi Feng, Anett Imolay-Kocsis, Paul B. Larson, Menachem Magidor, Jonathan Schilhan, Philipp Schlicht, Matteo Viale and W. Hugh Woodin for numerous illuminating conversations.

\section*{Use of Artificial Intelligence} 

No AI tools were used in the discovery and development of the results of this paper.

\vspace{1cm}

\noindent
\textsc{Márton Elekes}\\

\textsc{Alfréd Rényi Institute of Mathematics, Reáltanoda u. 13-15, H-1053 Budapest, Hungary and Eötvös Loránd University, Institute of Mathematics, Pázmány Péter s. 1/c, 1117 Budapest, Hungary}\\
\textit{Email address:} \texttt{elekes.marton@renyi.hu}\\
\textit{URL:} \texttt{http://www.renyi.hu/\textasciitilde emarci}

\bigskip

\noindent
\textsc{Máté Pálfy}\\

\textsc{Alfréd Rényi Institute of Mathematics, Reáltanoda u. 13-15, H-1053 Budapest, Hungary and Eötvös Loránd University, Institute of Mathematics, Pázmány Péter s. 1/c, 1117 Budapest, Hungary}\\
\textit{Email address:} \texttt{palfymateandras@gmail.com}

\end{document}